\documentclass{article}

\usepackage{siunitx}
\usepackage{setspace}
\usepackage{amsmath}
\usepackage{amssymb}
\usepackage{amsthm}
\usepackage{natbib}
\usepackage{xcolor}
\usepackage{graphicx}
\usepackage{mathtools}
\usepackage{url}
\usepackage{authblk}

\newcommand\joni[1]{{\color{black}{{#1}}}}

\newcommand\red[1]{{\color{black}{{#1}}}}

\usepackage[mathscr]{euscript}
\newcommand{\ten}[1]{\boldsymbol{\mathscr{#1}}}

\newcommand\numberthis{\addtocounter{equation}{1}\tag{\theequation}}

\newtheorem{assumption}{Assumption}
\newtheorem{theorem}{Theorem}
\newtheorem{lemma}{Lemma}
\newtheorem{remark}{Remark}
\newtheorem{definition_own}{Definition}

\newcommand{\tprime}{'}

\title{\red{Fast tensorial JADE}}

\begin{document}

\author[1,2]{Joni Virta}
\author[1]{Niko Lietz\'{e}n}
\author[1]{Pauliina Ilmonen}
\author[3]{Klaus Nordhausen}

\affil[1]{Aalto University School of Science, Finland}
\affil[2]{University of Turku, Finland}
\affil[3]{Vienna University of Technology, Austria}

\maketitle

\begin{abstract}
In this work, we propose a novel method for tensorial independent component analysis. Our approach is based on TJADE and $ k $-JADE, two recently proposed generalizations of the classical JADE algorithm. Our novel method achieves the consistency and the limiting distribution of TJADE under mild assumptions, and at the same time offers notable improvement in computational speed. Detailed mathematical proofs of the statistical properties of our method are given and, as a special case, a conjecture on the properties of $ k $-JADE is resolved. Simulations and timing comparisons demonstrate remarkable gain in speed. Moreover, the desired efficiency is obtained approximately for finite samples. The method is applied successfully to large-scale video data, for which neither TJADE nor $ k $-JADE is feasible. \red{Finally, an experimental procedure is proposed to select the values of a set of tuning parameters.}

\end{abstract}


\section{Tensorial independent component analysis}\label{sec:intro}
In modern data analysis an increasingly common and a natural assumption is that the covariance matrices are Kronecker-structured: the random vector $\textbf{x} \in \mathcal{R}^{pq}$ has a covariance matrix $\boldsymbol{\Sigma} = \boldsymbol{\Sigma}_2 \otimes \boldsymbol{\Sigma}_1$ where $\boldsymbol{\Sigma}_1 \in \mathcal{R}^{p \times p}$, $\boldsymbol{\Sigma}_2 \in \mathcal{R}^{q \times q}$ are positive-definite and $\otimes$ is the Kronecker product. In order for an  estimator of $\boldsymbol{\Sigma}$ to be adequate, it should utilize this special structure, see e.g.  \cite{ LengPan:2018,ros2016existence,srivastava2008models, werner2008estimation, wiesel2012geodesic}. In particular, if the Kronecker structure is ignored, the amount  of parameters is inflated from $p(p + 1)/2 + q(q + 1)/2$ to $pq(pq + 1)/2$.

Basic properties of the Kronecker product ensure that any zero-mean random vector $\textbf{x}$, with the Kronecker covariance structure, admits a natural representation as a random matrix by means of the matrix location-scatter model,
\[
\textbf{X} = \boldsymbol{\Sigma}_1^{1/2} \textbf{Z} \left( \boldsymbol{\Sigma}_2^{1/2} \right){}\tprime,
\]
where the original random vector is obtained from vectorization, $\mathrm{vec}[\textbf{X}] = \textbf{x}$, and the latent random matrix $\textbf{Z}$ is standardized such that $\mathrm{Cov} [ \mathrm{vec}[ \textbf{Z} ] ] = \textbf{I}_{pq}$. In order to find interesting structures, the uncorrelatedness assumption for the elements of $\textbf{Z}$ is occasionally not enough. In \cite{virta2017independent} the uncorrelatedness assumption was replaced by the assumption of full statistical independence and the location-scatter model was extended to the matrix-valued independent component model,
\begin{align}\label{eq:IC_model}
	\textbf{X} = \boldsymbol{\Omega}_1 \textbf{Z} \boldsymbol{\Omega}_2\tprime,
\end{align}
where the invertible $\boldsymbol{\Omega}_1 \in \mathcal{R}^{p \times p}, \boldsymbol{\Omega}_2 \in \mathcal{R}^{q \times q}$ are unknown parameters and the latent tensor $\textbf{Z}$ is assumed to have mutually independent, marginally standardized components. The estimation procedure of a latent vector with independent components, using only the information provided by observed linear combinations of the components, is referred to as independent component analysis (ICA). See \cite{ comon2010handbook,hyvarinen2004independent,ilmonen2011semiparametrically,  miettinen2015fourth, NordhausenOja:2018,samarov2004nonparametric} for different approaches under classic multivariate settings. Since the location has no effect on the estimation of the so-called mixing matrices $\boldsymbol{\Omega}_1$ and $\boldsymbol{\Omega}_2$, we can without loss of generality assume that $\textbf{Z}$ is centered. Under the additional assumption that at most one entire row of $\textbf{Z}$ has a multivariate normal distribution and at most one entire column of $\textbf{Z}$ has a multivariate normal distribution, it can be shown that the latent $\textbf{Z}$ is identifiable up to the order, joint scaling and signs of its rows and columns \citep{virta2017independent}. This pair of assumptions is less restrictive than its counterpart in standard vector-valued ICA, which allows a maximum of one normal component in the latent vector \citep{comon2010handbook}. On the contrary, the assumptions of the matrix model allow $\textbf{Z}$ to contain up to $pq - \mathrm{max}(p, q)$ normal components, if the non-normal elements are suitably located.

Vectorizing Eq. \eqref{eq:IC_model} yields a Kronecker-structured standard independent component model
\[
\textbf{x} = ( \boldsymbol{\Omega}_2 \otimes \boldsymbol{\Omega}_1 ) \textbf{z},
\]
where $\textbf{x} = \mathrm{vec}\left[\textbf{X}\right] \in \mathcal{R}^{pq}$, $\textbf{z} = \mathrm{vec}\left[\textbf{Z}\right] \in \mathcal{R}^{pq}$. Again, here any reasonable ICA approach should take the special form of the mixing matrix into account. To our knowledge, no structured estimation methods have yet been developed for standard vector-valued ICA in this setting. However, the problem has been considered under the matrix representation of Eq. \eqref{eq:IC_model} in \cite{virta2017independent,virta2017jade}, where two classical ICA procedures, \textit{fourth order blind identification} (FOBI) \citep{cardoso1989source} and \textit{joint diagonalization of eigen-matrices} (JADE) \citep{cardoso1993blind}, are extended to TFOBI and TJADE in order to solve Eq. \eqref{eq:IC_model}. Even though TJADE provides a definite improvement in efficiency with respect to TFOBI and the classical multivariate versions of the procedures, it is relatively expensive to compute. Consequently, the main goal of this paper is to address this issue. We provide a computationally more powerful, alternative TJADE-based estimator, which retains the desirable statistical properties of TJADE. In particular, the estimator retains the consistency and the limiting distribution of TJADE.

We wish to point out that the statistical models for matrix data we are about to discuss are markedly different from the ones usually encountered when discussing array-valued data in engineering contexts. The relevant engineering literature is populated mostly with different tensor decompositions, the most popular ones being the Tucker decomposition and the CP-decomposition, which extend the singular value decomposition to higher order structures in different ways, see \cite{cichocki2015tensor,kolda2009tensor} for comprehensive reviews. A fundamental difference between the classical tensor decompositions and our current model is that the former rarely incorporate the concept of sample in the formulation. As such, tensor decompositions generally reduce also the dimension of the observation space which is very unusual for classical statistical methods. These fundamentally different objectives also prevent any simple, meaningful comparisons, and therefore we have refrained from including comparisons to classical tensor decompositions in this work.

The rest of the paper is structured as follows. In Section \ref{sec:tjade}, we review the existing methodology on tensorial independent component analysis and, in particular, TJADE on which our proposed extension will be based. In Section \ref{sec:ktjade}, we formulate $k$-TJADE in the case of matrix-valued observations and thoroughly establish its limiting behavior along with the necessary assumptions. The theoretical results are extended 
to a general tensor-valued model in Section \ref{sec:tensors} using the technique of \textit{matricization}. Section \ref{sec:examples} explores the finite-sample performance of the proposed method under \red{both simulations and data examples on videos and hand-written digits. We also investigate the estimation of a set of tuning parameters of $k$-TJADE in the context of the latter example}. In Section \ref{sec:summary} we provide a short summary and list ideas for future work. Technical proofs are presented in Appendix \ref{sec:proofs}.

\section{Tensorial JADE}\label{sec:tjade}

We begin by briefly describing the theory behind TFOBI and TJADE. Everything in the following is formulated on the population level for a random matrix $\textbf{X} \in \mathcal{R}^{p \times q}$. In practice, one would obtain a sample of matrices, $\textbf{X}_1, \ldots , \textbf{X}_n$, from the distribution of $ \textbf{X} $ and the expected values below should be replaced by sample means. 

Assuming that the random matrix $\textbf{X} \in \mathcal{R}^{p \times q}$ follows the model in Eq. \eqref{eq:IC_model}, the first step is to simultaneously standardize both, the rows and the columns of the matrix, using the left and right covariance matrices of $\textbf{X}$,
\[
\boldsymbol{\Sigma}_1\left[\textbf{X}\right] = \frac{1}{q} \mathbb{E}\left[ \textbf{X} \textbf{X}\tprime \right] \quad \textnormal{and} \quad \boldsymbol{\Sigma}_2\left[\textbf{X}\right] = \frac{1}{p} \mathbb{E}\left[ \textbf{X}\tprime \textbf{X} \right].
\]
We denote the unique symmetric inverse square root of the positive-definite matrix $ \textbf{S} $ by $ \textbf{S}^{-1/2} $. The standardized variable $\textbf{X}^{\textnormal{st}} = ( \boldsymbol{\Sigma}^{-1/2}_1[\textbf{X}] ) \textbf{X} ( \boldsymbol{\Sigma}^{-1/2}_2[\textbf{X}] )\tprime$  satisfies $\textbf{X}^{\textnormal{st}} = \tau \textbf{U}_1 \textbf{Z} \textbf{U}_2\tprime$ for some $\tau = \tau(\boldsymbol{\Omega}_1, \boldsymbol{\Omega}_2) > 0$ and some orthogonal matrices, $\textbf{U}_1 \in \mathcal{R}^{p \times p}, \textbf{U}_2 \in \mathcal{R}^{q \times q}$, see \cite{virta2017independent}. The unknown constant of proportionality $\tau$ is a result of the joint scaling of the model in Eq. \eqref{eq:IC_model} being left unfixed. After the standardization,  solving  the IC problem is reduced to the estimation of the orthogonal matrices $\textbf{U}_1, \textbf{U}_2$, a task commonly addressed in ICA using higher-order cumulants. The TFOBI and TJADE procedures also utilize the higher-order cumulants. In TFOBI, a Fisher consistent (under mild assumptions, see Section \ref{sec:ktjade})  estimator, $\boldsymbol{\Gamma}^\textnormal{F}[\textbf{X}]$,  for the inverse of the matrix $\boldsymbol{\Omega}_1$ is constructed such that 
\begin{align*}
\boldsymbol{\Gamma}^\textnormal{F}[\textbf{X}] = \left(\textbf{V}{}^\textnormal{F}[\textbf{X}]\right)\tprime \boldsymbol{\Sigma}_1^{-1/2}[\textbf{X}],
\end{align*}
where the columns of  $\textbf{V}^\textnormal{F}[\textbf{X}]$ are the eigenvectors of the matrix
\begin{align*}
\textbf{B}\left[\textbf{X}\right] = \frac{1}{q} \mathbb{E} \left[ \textbf{X} \textbf{X}\tprime \textbf{X} \textbf{X}\tprime \right].
\end{align*}
Thus, $\boldsymbol{\Gamma}^\textnormal{F}[\textbf{X}]$ provides a solution for the left-hand side of the IC model in Eq. \eqref{eq:IC_model} \citep{virta2017independent}.

The TJADE procedure utilizes a set of matrices, $\mathcal{C} = \left\{ \textbf{C}^{ij}\left[\textbf{X}^\textnormal{st}\right]  : i, j \in \left\{ 1, \ldots , p \right\}  \right\}$, referred to as the set of cumulant matrices, such that
\begin{align}\label{eq:cumulant_matrix} 
\textbf{C}^{ij}\left[\textbf{X}\right] = \frac{1}{q} \mathbb{E} \left[ \textbf{e}_i\tprime \textbf{X} \textbf{X}\tprime \textbf{e}_j  \textbf{X} \textbf{X}\tprime \right] - \boldsymbol{\Sigma}_1\left[\textbf{X}\right] \left( \delta_{ij} q \textbf{I}_p + \textbf{E}^{ij} + \textbf{E}^{ji} \right) \left(\boldsymbol{\Sigma}_1\left[\textbf{X}\right]\right)\tprime,
\end{align}
where $ \delta_{ij} $ is the Kronecker delta, $\textbf{e}_k \in \mathcal{R}^p$, $k \in \{ 1, \ldots , p \}$, are the standard basis vectors of $\mathcal{R}^p$ and $\textbf{E}^{kl} = \textbf{e}_k \textbf{e}_l\tprime$. The left covariance matrix of the standardized matrix satisfies $\boldsymbol{\Sigma}_1[\textbf{X}^\textnormal{st}] = \tau^2 \textbf{I}_p$ and is included in Eq. \eqref{eq:cumulant_matrix} solely for the estimation of the constant of proportionality $\tau$. The authors in \cite{virta2017jade} proved that the \textit{joint diagonalizer} of $\mathcal{C}$ is under mild assumptions, see Section \ref{sec:ktjade}, equal to the orthogonal matrix $\textbf{U}_1$ up to the order and signs of its columns. The joint diagonalizer of the set $\mathcal{C}$ is defined as any orthogonal matrix $\textbf{V} \in \mathcal{R}^{p \times p}$ that minimizes
\begin{align}\label{eq:joint_diag}
\tilde{g}\left(\textbf{V},\textbf{X}^\textnormal{st}\right) = \sum_{i,j= 1}^p  \left\| \mathrm{off} \left( \textbf{V}\tprime \textbf{C}^{ij}\left[\textbf{X}^\textnormal{st}\right] \textbf{V} \right)  \right\|_\textnormal{F}^2,
\end{align}
where $\mathrm{off}(\textbf{S}) \in \mathcal{R}^{p \times p}$ is equal to $\textbf{S} \in \mathcal{R}^{p \times p}$ with the diagonal elements set to zero.

The joint diagonalizer defines a coordinate system in which the linear transformations $\textbf{C}^{ij}[\textbf{X}^\textnormal{st}]$, $i, j \in \left\{ 1, \ldots , p \right\}$, \red{have minimal sum of squared off-diagonal elements}. There exists several algorithms for optimizing Eq. \eqref{eq:joint_diag}, the most popular being the Jacobi-rotation technique, for details see \cite{belouchrani1997blind,IllnerMiettinenFuchsTaskienNordhausenOjaTheis2015}. After the estimation of the joint diagonalizer $\textbf{V}^\textnormal{J}\left[\textbf{X}\right]$, an estimated inverse for the matrix $\boldsymbol{\Omega}_1$ is  obtained as the TJADE-fucntional, $\boldsymbol{\Gamma}^\textnormal{J}\left[\textbf{X}\right] = (\textbf{V}^\textnormal{J}\left[\textbf{X}\right])\tprime \boldsymbol{\Sigma}_1^{-1/2}\left[\textbf{X}\right]$.

The results of this paper are derived only for the left-hand side of the matrix-valued model. The right-hand side of the matrix-valued model can be solved exactly as the left-hand side. One can simply take the transpose of $\textbf{X}$ and proceed as with the left-hand side. Only the sizes of the cumulant and transformation matrices change from $p \times p$ to $q \times q$. Moreover, matricization allows us to extend the estimators beyond the matrix-valued model to arbitrary-dimensional tensor-valued IC models, see \cite{virta2017independent,virta2017jade}. Matricization allows us to hide a considerable amount of the unpleasant notation related to tensor algebra, see Section \ref{sec:tensors}. In total, it is sufficient to present the results in matrix form and only for the left-hand side of the model in Eq. \eqref{eq:IC_model}.


When the dimension $q$ is equal to one, the TJADE procedure for the left-hand side of the model is equivalent to the standard JADE for vector-valued data. Extensive comparisons between JADE and TJADE are conducted in \cite{VirtaTaskinenNordhausen:2016,virta2017independent,virta2017jade} with the conclusion that the latter is uniformly superior to the former under the Kronecker-structured IC model.  Moreover, the tensorial version is computationally significantly faster. Consider a tensor of $r$th order with all dimensions of size $p$. Standard JADE requires a single joint diagonalization of $p^{2r}$ matrices that are of size $p^r \times p^r$, whereas TJADE requires $r$ joint diagonalizations of $p^2$ matrices that are of size $p \times p$. In essence, adding dimensions to a tensor has a multiplicative effect on the number of operations  the classic vectorial methods require and merely an additive effect on the tensorial methods. However, even with its considerable advantages over JADE, running the TJADE procedure is slow for large tensors.

To obtain a faster method, we approach the problem in the spirit of \cite{miettinen2013fast} where a faster version of JADE, $ k $-JADE, is derived. The modification can be described very succinctly: instead of diagonalizing the entire set of cumulant matrices $\textbf{C}^{ij}$, we diagonalize only a specific subset of them, chosen such that the desirable statistical properties of TJADE are still carried over to the extension. Since the subset of cumulant matrices can be chosen separately in each direction of the tensor, the $ k $-JADE approach provides even more significant improvements in tensorial settings compared to its original use in improving JADE. Note that, similar ideas as in \cite{miettinen2013fast} were used already in \cite{cardoso1999high} to formulate \textit{shifted blocks for blind separation} (SHIBBS), where only those cumulant matrices of regular JADE with matching indices are diagonalized.


\section{Tensorial $k$-JADE}\label{sec:ktjade}

In this section we propose a novel extension of the TJADE procedure. We formulate the extension, $ k $-TJADE, such that it retains the following three key properties of TJADE. The first of these properties is the ability to solve the tensor independent component model, manifesting either as Fisher consistency or consistency, depending on whether we are at the population or sample level, respectively. The second property is orthogonal equivariance under arbitrary data. ICA-estimators are customarily expected to have some form of equivariance, which makes the generalization of limiting properties more straightforward \citep{miettinen2015fourth, virta2017independent}.
The third desired property is the limiting distribution of TJADE, the \red{one with the lowest limiting variance of the known tensorial ICA methods.} Next, we establish these properties one-by-one. As mentioned in the previous section, all results derived for the left-hand side of the model also hold for the right-hand side, prompting us to consider only the former in the following. \red{In the same spirit, even though the proposed $ k $-TJADE method has in the case of matrix data two tuning parameters, $ k_1 $ for the rows and $ k_2 $ for the columns, the method still acts only on a single mode at a time, and as such we omit the subscripts and speak in the following only of the tuning parameter $ k $. Moreover, the same idea is reflected in the name of the method which should (for matrix data) technically be $ (k_1, k_2) $-TJADE. In order to keep the presentation more readable, we prefer to call the method simply $ k $-TJADE.}

We define \textit{matrix independent component functionals}, the extension of independent component functionals \citep{miettinen2015fourth} to matricial ICA.
\begin{definition_own}\label{def:ic_functional}
	A $p \times p$ matrix-valued functional $\boldsymbol{\Gamma}$ is a matrix independent component (IC) functional if
	\begin{itemize}
		\item[(i)] $\boldsymbol{\Gamma}\left[\textbf{X}\right] \equiv \boldsymbol{\Omega}_1^{-1}$ for all $\textbf{X} \in \mathcal{R}^{p \times q}$ that follow the matrix IC model of Eq. \eqref{eq:IC_model},
		\item[(ii)] $\boldsymbol{\Gamma}\left[\textbf{U}_1 \textbf{X} \textbf{U}_2\tprime\right] \equiv \boldsymbol{\Gamma}\left[\textbf{X}\right] \textbf{U}_1\tprime$ for all $\textbf{X} \in \mathcal{R}^{p \times q}$ and all orthogonal $\textbf{U}_1 \in \mathcal{R}^{p \times p}$, $\textbf{U}_2 \in \mathcal{R}^{q \times q}$,
	\end{itemize}
	where two matrices $\textbf{A}, \textbf{B} \in \mathcal{R}^{p \times p}$ satisfy $\textbf{A} \equiv \textbf{B}$ if $\textbf{A} = c \textbf{P} \textbf{J} \textbf{B}$ for some $c > 0$, some diagonal matrix $\textbf{J} \in \mathcal{R}^{p \times p}$ with diagonal elements equal to $\pm 1$ and some permutation matrix $\textbf{P} \in \mathcal{R}^{p \times p}$.
\end{definition_own}

The first condition in Definition \ref{def:ic_functional} requires that a matrix IC functional must be able to solve the left-hand side of the model in Eq. \eqref{eq:IC_model} (Fisher consistency). The second condition essentially states that the functional cancels out any orthogonal transformations on the observed matrices (orthogonal equivariance). As a particularly useful consequence of the latter, the limiting distribution of a matrix IC functional under trivial mixing, $\boldsymbol{\Omega}_1 = \textbf{I}_p$, $\boldsymbol{\Omega}_2 = \textbf{I}_q$, instantly generalizes to any orthogonal mixing as well.

Let $\boldsymbol{\kappa} \in \mathcal{R}^p$ be the vector of the row means of the element-wise kurtoses, $\mathbb{E}\left[x_{kl}^4\right] - 3$, of the elements of $\textbf{Z}$.
\begin{assumption}\label{assu:JADE}
	At most one element of $\boldsymbol{\kappa}$ equals zero.
\end{assumption}

\begin{assumption}[$ v $]\label{assu:FOBI}
	The multiplicities of the elements of $\boldsymbol{\kappa}$ are at most $v$.
\end{assumption}
The TFOBI functional $ \boldsymbol{\Gamma}^\textnormal{F} $ is a matrix IC functional in the sense of Definition \ref{def:ic_functional} if Assumption \ref{assu:FOBI}$(1)$ is satisfied and the TJADE functional $\boldsymbol{\Gamma}^\textnormal{J}$ is a matrix IC functional if Assumption \ref{assu:JADE} is satisfied. Naturally, the column mean analogues of the assumptions are required to separate the right-hand side of the model in Eq. \eqref{eq:IC_model}.

The same comparison as was done between the normality assumptions in Section \ref{sec:intro} holds analogously between Assumption \ref{assu:JADE}, Assumption \ref{assu:FOBI}$ (\red{\nu}) $ and their vectorial counterparts allowing maximally one zero-kurtosis component or only distinct kurtoses, respectively. The main implication is that in matrix ICA numerous latent components may have identical kurtoses as long as their row means (and column means when separating the right-hand side of the model) satisfy the necessary requirements. The assumptions also satisfy the following set of relations,
\[ 
\mbox{\textbf{Assumption} }2(1) \subset \cdots \subset \mbox{\textbf{Assumption} }2(p),
\]
where $ \subset $ means ``implies''. Moreover, we have also $\mbox{\textbf{Assumption} }2(1) \subset \mbox{\textbf{Assumption} }1$.

In order to speed up TJADE such that the properties in Definition \ref{def:ic_functional} are retained, we proceed as in \cite{miettinen2013fast}, and instead of diagonalizing the set $\mathcal{C}$, we diagonalize only those members of it which satisfy $|i - j| < k$, for some pre-defined value of the tuning parameter $k \in \{1, \ldots , p \}$. This discarding can be motivated in two ways. Firstly, all except the repeated index matrices, $\textbf{C}^{11}, \ldots , \textbf{C}^{pp}$, vanish asymptotically. Every matrix $\textbf{C}^{ij}\left[\textbf{Z}\right] = \textbf{0}$, $i \neq j$, implying that with increasing sample size all the separation information is eventually contained in the $ p $ repeated index matrices. Secondly, by assuming that the values in $\boldsymbol{\kappa}$ are in decreasing order (this is guaranteed by \red{using TFOBI as a preprocessing step, see the next paragraph}) the $i$th row of $\textbf{Z}$ is the most difficult to separate from its immediate neighboring rows and the separation information between them is contained precisely in the matrices $\textbf{C}^{ij}$ and $\textbf{C}^{ji}$ where $j$ is close to $i$.

Analogously to $k$-JADE, we use the TFOBI-algorithm to obtain an initial value for the functional. This ensures that even after the previous modification, the functional remains orthogonally equivariant. The following definition and theorem formalize our resulting novel method, called $k$-TJADE.

\begin{definition_own}
	Fix $k \leq p$. The $k$-TJADE functional is \[\boldsymbol{\Gamma}^{k} [ \textbf{X} ] = (\textbf{V} [ \textbf{X}^\textnormal{F} ])\tprime \boldsymbol{\Gamma}^\textnormal{F}\left[\textbf{X}\right],\] where $\boldsymbol{\Gamma}^\textnormal{F}$ is the TFOBI functional, $\textbf{X}^\textnormal{F} = \boldsymbol{\Gamma}^\textnormal{F}\left[\textbf{X}\right] \textbf{X} (\boldsymbol{\Gamma}^\textnormal{F}[\textbf{X}\tprime])\tprime$ is the TFOBI-solution for $\textbf{X}$ and the orthogonal matrix $\textbf{V} [ \textbf{X}^\textnormal{F} ] = (\textbf{v}_1, \ldots , \textbf{v}_p)$ is the joint diagonalizer of the set of matrices
	$
	\mathcal{C}_k = \{ \textbf{C}^{ij} [ \textbf{X}^\textnormal{F} ] : \lvert i - j \rvert < k \}.
	$
\end{definition_own}

\begin{theorem}\label{theo:functional}
	Let Assumptions \ref{assu:JADE} and \ref{assu:FOBI}$(v)$ hold for some fixed $v$. Then the $k$-TJADE functional $\boldsymbol{\Gamma}^k$ is a matrix IC functional for all $k \geq v$.
\end{theorem}

Theorem \ref{theo:functional} provides the Fisher consistency and the orthogonal equivariance for the $ k $-TJADE functional. The assumptions that Theorem \ref{theo:functional} requires are interesting since they provide an interpretation for the tuning parameter $ k $ --- the parameter is the maximal number of allowed   kurtosis mean multiplicities. The values $ k = 1 $ and $ k = p $ correspond to the extreme cases where all the kurtosis means have to be distinct (as in TFOBI) and where no assumptions are made on the multiplicities of the non-zero kurtosis means (as in TJADE). Thus, $ k $-TJADE can be seen as a middle ground between TFOBI and TJADE. As the assumptions, also the methods can be ordered according to the strictness of the assumptions they require,
\[ 
\mbox{TFOBI} \succeq  1 \mbox{-TJADE} \succeq  \cdots \succeq   p \mbox{-TJADE} = \mbox{TJADE},
\] 
where ``$ \succeq  $'' is read as ``makes at least as many assumptions as'' and ``$ = $'' as ``makes the same assumptions as''. 

\red{We next give some intuition behind the proof of Theorem \ref{theo:functional} (given in Appendix \ref{sec:proofs}). The proof relies on a specific interplay between the preliminary TFOBI-step and the set of cumulant matrices $ \mathcal{C} $. Being an eigendecomposition-based method, TFOBI is able to estimate the matrix $ \boldsymbol{\Omega}_1^{-1} $ only up to row blocks determined by the elements of $ \boldsymbol{\kappa} $, such that the rows of the estimate corresponding to equal kurtosis values remain mixed by orthogonal matrices (c.f. eigenvectors corresponding to multiple eigenvalues are not uniquely defined but their span is). However, the proof also shows that the joint diagonalization of all cumulant matrices with $ |i - j| < k $, for some given $ k $, can separate mixed row blocks of size at most $ k $ (intuitively, increasing $ k $ means using more cumulant matrices which increases the amount of available information). Thus, by sequencing the two steps together and having $ k \geq \nu $ ensures that any blocks left still mixed after the TFOBI-step will be unmixed in the joint diagonalization step.}

Assumptions \ref{assu:JADE} and \ref{assu:FOBI}$(v)$ not only provide Theorem \ref{theo:functional}, but are also sufficient to guarantee the two remaining desired properties of $ k $-TJADE, consistency and the same limiting distribution as that of TJADE. These asymptotic properties are formalized in the following two theorems. Remarkably as a special case, when $ q = 1 $, the latter also proves a previously unsolved conjecture posed about the limiting behavior of vectorial $ k $-JADE in \cite{miettinen2015fourth}.

\begin{theorem}\label{theo:consistency}
	Let Assumptions \ref{assu:JADE} and \ref{assu:FOBI}$(v)$ hold for some fixed $ v $ and let $\textbf{X}_1, \ldots , \textbf{X}_n$ be an i.i.d. sequence from the matrix IC model in Eq. \eqref{eq:IC_model} with identity mixing matrices, $\boldsymbol{\Omega}_1 = \textbf{I}_p, \boldsymbol{\Omega}_2 = \textbf{I}_q$. Assume that the population quantity $\textbf{X}$ has finite eight moments. Then, for all $k \geq v$, there exists a consistent sequence of $ k $-TJADE-estimators (indexed by $ n $). That is,
	\[
	\hat{\boldsymbol{\Gamma}}{}^{k}  \rightarrow_{\mathbb{P}} \textbf{I}_{p}.
	\]
\end{theorem}

\begin{theorem}\label{theo:limiting}
	Under the assumptions of Theorem \ref{theo:consistency}, we have for all $ k \geq v $ that,
	\[
	\sqrt{n} (\hat{\boldsymbol{\Gamma}}{}^{k} - \textbf{I}_{p} ) = \sqrt{n} (\hat{\boldsymbol{\Gamma}}{}^\textnormal{J} - \textbf{I}_{p} ) + o_p(1),
	\]
	where $\hat{\boldsymbol{\Gamma}}{}^\textnormal{J}$ is the TJADE-estimator. The notation $ o_p(1) $ refers to a sequence of random matrices that converges in probability to the zero matrix.
\end{theorem}

\red{The proofs of Theorem \ref{theo:consistency} and \ref{theo:limiting} are based on the fact that the kurtosis matrices of the latent $ \textbf{Z} $ with non-matching indices vanish asymptotically, $\textbf{C}^{ij}[\textbf{Z}] = \textbf{0}$, $i \neq j$. The key point of the proof is then to show that this property transfers to the corresponding kurtosis matrices $ \textbf{C}^{ij}[\textbf{X}^\textnormal{F}] $, of the TFOBI-standardized observation, to such extent that the asymptotical contribution of the sample estimates of these matrices to the final $ k $-TJADE estimate is negligible. Thus, regardless of the choice of $ k $, the asymptotical behavior of the method is determined by the $ p $ repeated index matrices $\textbf{C}^{11}, \ldots , \textbf{C}^{pp}$. Finally, the proof reveals that the rotation specified by the joint diagonalization dominates over the preliminary TFOBI-rotation and the effect of the latter also turns out to be asymptotically negligible.}

Note that, $\hat{\boldsymbol{\Gamma}}{}^\textnormal{J} = \hat{\boldsymbol{\Gamma}}{}^{p}$ does not generally hold as the latter estimator utilizes the preliminary TFOBI-step while the former does not. The limiting distribution of $ \sqrt{n} (\hat{\boldsymbol{\Gamma}}{}^\textnormal{J} - \textbf{I}_{p} )  $ is in \cite{virta2017jade} shown to be multivariate normal and closed form expressions for its limiting variances are also given therein. By the orthogonal equivariance of matrix independent component functionals (property (ii) of Definition \ref{def:ic_functional}) the limiting results of Theorem \ref{theo:limiting} generalize to any orthogonal mixing matrices. Note that the original $ k $-JADE in \cite{miettinen2013fast} is affine equivariant (equivariant under all coordinate system changes, not just orthogonal). The problem of achieving affine equivariance in the context of tensorial ICA is discussed in \cite{virta2017jade}. There it was conjectured that tensorial ICA cannot be affine equivariant.

The two limiting theorems above show that, under suitable assumptions, $ k $-TJADE indeed has all the desirable properties listed at the beginning of this section. This can be summarized by saying that $ k $-TJADE makes a trade-off between assumptions and computational load: with the price of added assumptions, we obtain a method with the same limiting efficiency as TJADE, but with significantly lighter computational burden. As the claim about efficiency holds only asymptotically, we conduct a simulation study (in Section \ref{sec:examples}) to compare the finite-sample efficiency of the estimators.

\red{Note that the maximal kurtosis multiplicity $ \nu $ is unknown in practice, which makes the choosing of $ k $ such that $ k \geq \nu $ a non-trivial task. However, the simulations of Section \ref{sec:examples} reveal that not a lot of separation efficiency is necessarily lost when using a slightly too small value of $ k $, making the problem less dire. To further alleviate the issue, in the final part of Section \ref{sec:examples}, we propose a procedure for estimating $ \nu $ and apply it to hand-written digit data.}

\section{A note on tensorial ICA}\label{sec:tensors}

In this section, we formulate the general tensorial IC model and discuss how it can be reduced to the matricial IC model. We begin with a short review of the basic concepts of multilinear algebra.

An $ r $th order tensor $ \ten{X} = (x_{i_1 \cdots i_r}) \in \mathcal{R}^{p_1 \times \cdots \times p_r}$ is an $ r $-dimensional array containing a total of $ \rho = \prod_{m = 1}^r p_m $ elements and has a total of $ r $ \textit{modes} or \textit{ways}, that is, directions from which we can view it. For example, a matrix ($ r = 2 $) can be viewed either through its columns or through its rows. Two complementary ways of dividing a tensor into a disjoint collection of smaller tensors are called the $ m $-mode vectors and the $ m $-mode faces. In the former, we choose an index $ m = 1, \ldots , r$ and have the values of the indices $ \{1, \ldots , r \} \setminus \{ m \} $ fixed and let the $ m $th index vary over its range. Each fixed combination of the $ r - 1 $ indices then yields a single $ p_m $-dimensional vector and the collection of all $ \rho/p_m $ such vectors is called the set of $ m $-mode vectors of $ \ten{X} $. On the other hand, if we fix the value of the $ m $th index and let the others vary over their ranges, we get a total of $ p_m $ tensors of order $ r - 1 $ and size $ p_1 \times \cdots \times p_{m-1} \times p_{m+1} \times \cdots \times p_r $, called the $ m $-mode faces of $ \ten{X} $. Illustrations of both, $ m $-mode vectors and $ m $-mode faces, $ m = 1, 2, 3 $, in the case of a 3-dimensional tensor, are shown in Figures \ref{fig:tensor_modes} and \ref{fig:tensor_faces}.

\begin{figure}[t]
	\begin{center}
		\includegraphics[width=1\textwidth]{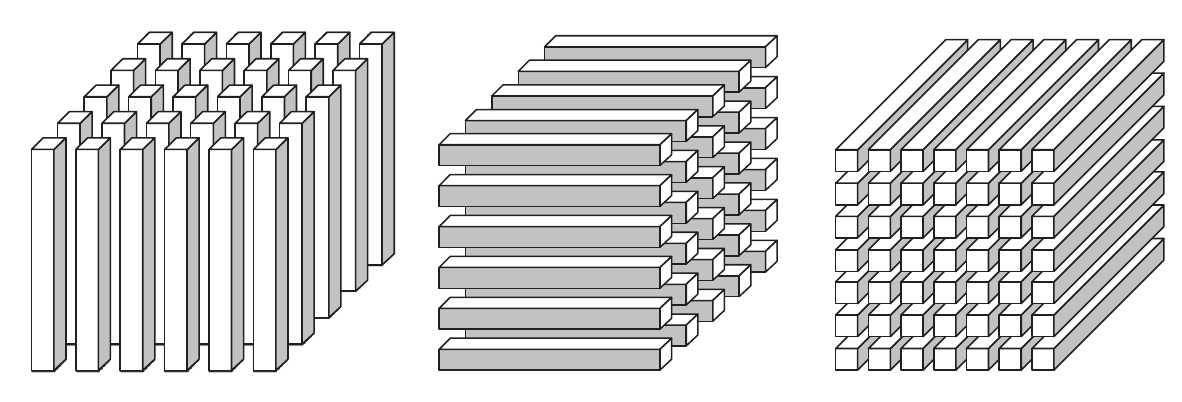}
	\end{center}
	\caption{The 1-mode, 2-mode and 3-mode vectors of a 3-dimensional tensor. Elina Vartiainen\textsuperscript{\copyright}. }
	\label{fig:tensor_modes}
\end{figure}

\begin{figure}[t]
	\begin{center}
		\includegraphics[width=1\textwidth]{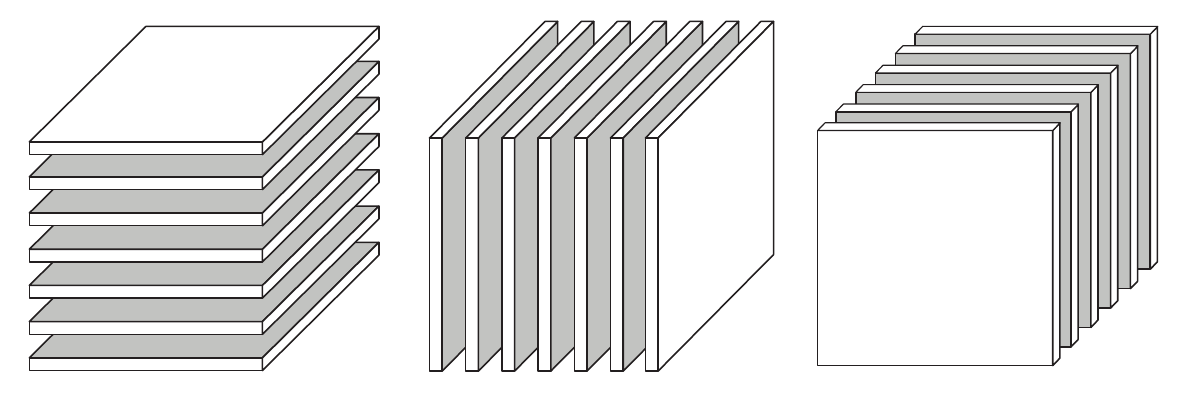}
	\end{center}
	\caption{The 1-mode, 2-mode and 3-mode faces of a 3-dimensional tensor.  Elina Vartiainen\textsuperscript{\copyright}.}
	\label{fig:tensor_faces}
\end{figure}

Two algebraic operations can now be defined in terms of the $ m $-mode vectors. Let $ \ten{X} \in \mathcal{R}^{p_1 \times \cdots \times p_r} $ and $ \textbf{A}_m = (a^{(m)}_{i_m j_m}) \in \mathcal{R}^{q_m \times p_m} $ for some fixed $ m = 1, \ldots , r $. The $ m $-mode multiplication $ \ten{X} \times_m \textbf{A}_m \in \mathcal{R}^{p_1 \times \cdots \times p_{m - 1} \times q_m \times p_{m + 1} \times \cdots \times p_r } $ of $ \ten{X} $ from the $ m $th mode by $ \textbf{A}_m $ is defined element-wise as,
\[ 
(\ten{X} \times_m \textbf{A}_m)_{i_1, \ldots, i_r} = \sum_{j_m = 1}^{p_m} x_{i_1 \cdots i_{m - 1} j_m i_{m + 1} \cdots i_r} a^{(m)}_{i_m j_m}. 
\]
The $ m $-mode multiplication is easily understood: $\ten{X} \times_m \textbf{A}_m$ is obtained by multiplying each $ m $-mode vector of $ \ten{X} $ from the left by $ \textbf{A}_m $ and collecting the resulting vectors back into an $ r $-dimensional tensor in the same order. The $ m $-mode multiplications $ \times_1 , \ldots , \times_r$ from distinct modes are commutative and we use the shorthand $ \ten{X} \times_{m=1}^r \textbf{A}_m = \ten{X} \times_1 \textbf{A}_1 \cdots \times_r \textbf{A}_r $. In the case of matrices, $ \ten{X} = \textbf{X} \in \mathcal{R}^{p_1 \times p_2} $ $(r = 2)$, the simultaneous multiplication from both modes simply gives $ \textbf{X} \times_{m=1}^2 \textbf{A}_m = \textbf{A}_1 \textbf{X} \textbf{A}_2\tprime$.

To show the connection between the matricial and tensorial IC methods, we still need the concept of $ m $-mode matricization. For a fixed mode $ m = 1 \ldots ,r $, the $ m $-mode matricization $ \textbf{X}^{(m)} \in \mathcal{R}^{p_m \times \rho/p_m} $ of $ \ten{X} $ is obtained by taking the $ m $-mode vectors of $ \ten{X} $ and collecting them horizontally into a wide matrix. The arrangement order of the $ m $-mode vectors has no effect on our further developments, and we choose to use the cyclical ordering as in \cite{de2000multilinear}. In this case the relationship,
\begin{align}\label{eq:matricization}
(\ten{X} \times_m \textbf{A}_m)^{(m)} = \textbf{A}_m \textbf{X}^{(m)} (\textbf{A}_{m + 1} \otimes \cdots \otimes \textbf{A}_r \otimes \textbf{A}_1 \otimes \cdots \otimes \textbf{A}_{m - 1})\tprime
\end{align}
holds. For more comprehensive introduction to multilinear algebra, see \cite{ cichocki2009nonnegative, de2000multilinear}.

We are now sufficiently equipped to examine the connection between the matricial ICA and the tensorial ICA. The zero-mean $ r $th order random tensor $ \ten{X} \in \mathcal{R}^{p_1 \times \cdots \times p_r}$ is said to follow the tensorial independent component model if
\begin{align}\label{eq:tensor_model} 
\ten{X} = \ten{Z} \times_{m=1}^r \boldsymbol{\Omega}_m,
\end{align}
where the random tensor $\ten{Z} \in \mathcal{R}^{p_1 \times \cdots \times p_r}$ has mutually independent, marginally standardized components and $ \boldsymbol{\Omega}_m \in \mathcal{R}^{p_m \times p_m}$, $ m = 1, \ldots ,r $, are invertible \citep{virta2017independent}. We further assume that for each mode $ m = 1, \ldots ,r$ at most one $ m $-mode face of $ \ten{Z} $ consists solely of normally distributed components. The objective of tensorial independent component analysis is to estimate $ \ten{Z} $ given a random sample of observed tensors from the distribution of $ \ten{X} $.

We fix the mode $ m = 1, \ldots , r$ and consider the $ m $-mode matricization of the model in Eq. \eqref{eq:tensor_model}. It now follows from Eq. \eqref{eq:matricization} that
\begin{align}\label{eq:tensor_model_matricization}
\textbf{X}^{(m)} =  \boldsymbol{\Omega}_m \textbf{Z}^{(m)} (\boldsymbol{\Omega}_{m + 1} \otimes \cdots \otimes \boldsymbol{\Omega}_r \otimes \boldsymbol{\Omega}_1 \otimes \cdots \otimes \boldsymbol{\Omega}_{m - 1})\tprime.
\end{align}
As Kronecker products of invertible matrices are themselves invertible, a comparison to Eq. \eqref{eq:IC_model} now reveals that Eq. \eqref{eq:tensor_model_matricization} is in the form of a matrix IC model with $ \boldsymbol{\Omega}_m $ replacing $ \boldsymbol{\Omega}_1 $ and $ \boldsymbol{\Omega}_{m + 1} \otimes \cdots \otimes \boldsymbol{\Omega}_r \otimes \boldsymbol{\Omega}_1 \otimes \cdots \otimes \boldsymbol{\Omega}_{m - 1} $ taking the role of $ \boldsymbol{\Omega}_2 $. In addition, $ \textbf{Z}^{(m)} $ satisfies all the assumptions of the matrix IC model. \red{To be precise, it is possible that multiple columns of $ \textbf{Z}^{(m)} $ have multivariate normal distributions as a consequence of the matricization. However, as the $ m $-mode matricization leaves the structure of the $ m $th mode intact, the normality assumption on $ \ten{Z} $ guarantees that at most one row of $ \textbf{Z}^{(m)} $ has a multivariate normal distribution, which is then sufficient for the identifiability of $ \boldsymbol{\Omega}_m $, the current parameter of interest.} Thus, using Eq. \eqref{eq:tensor_model_matricization}, we can estimate $ \boldsymbol{\Omega}_m $ exactly as $ \boldsymbol{\Omega}_1 $ in the matricial case. In the tensorial case, the kurtosis means in the vector $ \boldsymbol{\kappa} $ in Assumption \ref{assu:JADE} and in Assumption \ref{assu:FOBI}$ (v) $ are computed over the rows of $ \textbf{Z}^{(m)} $, or equivalently, over the $ m $-mode faces of $ \ten{Z} $. All our theoretical results, such as orthogonal equivariance (a Kronecker product of orthogonal matrices is itself orthogonal) and the asymptotic variances, hold fully under the tensorial IC model, as long as the relevant kurtosis assumptions are satisfied. \red{Finally, we remind that for an $ r $th order input tensor $ \ten{X} $, the $ k $-TJADE has a total of $ r $ tuning parameters, $ k_1, \ldots , k_r $, with the underlying idea that $ k_m $ should be chosen to be equal to or larger than the maximal kurtosis multiplicity in the corresponding mode, see Theorem \ref{theo:limiting}.}

\section{Simulations and examples}\label{sec:examples}
In this section, we illustrate the finite sample properties of  the tensor-valued procedures TFOBI, TJADE and the novel $k$-TJADE, which are implemented in the R-package \textit{tensorBSS} \citep{virta2016tensorbss}. For comparison, we also consider the classical vector-valued versions of these estimators, denoted by VFOBI, VJADE and $k$-VJADE, as implemented in the JADE package \citep{miettinen2017blind}.  In the classical procedures, the tensor-valued observations are first vectorized and the algorithms are then applied to the resulting data matrices. 

Next, we consider a collection of observed i.i.d. tensors  $\{\ten{X}_j\}_{j\in\{1,\ldots n\}}$, generated from the tensorial independent component model, such that $\ten{X}_j \in \mathcal{R}^{p_1\times \cdots \times p_r}$, for every $j$. Let $\boldsymbol{\Omega}_1, \ldots, \boldsymbol{\Omega}_r$ be the theoretical mixing matrices and let $\hat{\boldsymbol{\Gamma}}_1, \ldots, \hat{\boldsymbol{\Gamma}}_r$ be the corresponding unmixing estimates produced by one of the tensor-valued procedures. We denote the Kronecker products of the matrices as $\hat{\boldsymbol{\Gamma}} = \hat{\boldsymbol{\Gamma}}_1 \otimes \ldots \otimes \hat{\boldsymbol{\Gamma}}_r $ and $\boldsymbol{\Omega} = \boldsymbol{\Omega}_1 \otimes \ldots \otimes \boldsymbol{\Omega}_r $. The vector-valued procedures produce a single unmixing estimate, denoted also by $\hat{\boldsymbol{\Gamma}}$.
Note that the estimates produced by the vectorial and the tensorial methods are comparable, since in both cases the matrix $\hat{\boldsymbol{\Gamma}}$ estimates the inverse of the compound matrix $\boldsymbol{\Omega}$.

The so-called gain matrix is defined as $\hat{\textbf{G}} = \hat{\boldsymbol{\Gamma}}   \boldsymbol{\Omega} $.  The unmixing is considered successful if the gain matrix is close to the identity matrix, up to the order and signs of its rows.  We quantify this closeness using the performance measure called minimum distance (MD) index \citep{ilmonen2010new}. The MD index is formulated as follows,
\begin{align}
\label{eq:mdindex}
D(\hat{\textbf{G}}) = \frac{1}{\sqrt{\rho - 1}} \underset{\textbf{C} \in \mathcal{C}_0}\inf \left\| \textbf{C} \hat{\textbf{G}} - \textbf{I}_{\rho} \right\|_\textnormal{F},
\end{align}
where $\rho = \prod_{j=1}^r p_j $ and  $ \mathcal{C}_0 $ is the set of all matrices with exactly one non-zero element in each row and column. The range of the MD index is $ [0, 1] $, where the value $ 0 $ corresponds to the case of the gain matrix being exactly a permuted and scaled identity matrix, i.e. the estimate is perfectly accurate. Additionally, the limiting distribution of the MD index   can be obtained from the limiting distribution of the corresponding IC functional $ \hat{\boldsymbol{\Gamma}} $, see \cite[Theorem 1]{ilmonen2010new} and \cite[Theorem 6]{virta2017independent}. 

In simulation studies, where multiple iterations are performed under identical conditions, the asymptotic value for the mean of the transformed MD index $ n (\rho - 1) D(\hat{\textbf{G}})^2$ can be obtained using the limiting variances of  the applied IC functionals, see  \cite{virta2017independent} for further details. The convergence towards the theoretical limiting values given  by Theorem \ref{theo:limiting} can then be demonstrated by  visualizing the theoretical limits alongside $\ell_i^{-1}\sum_{j=1}^{\ell_i} n_i (\rho - 1) D(\hat{\textbf{G}}_j)^2$, where $\ell_i$ is the number of  iterations for the sample size $n_i$ and $ \hat{\textbf{G}}_{j} $ is the estimated gain matrix for the corresponding $ j $th iteration.

Before presenting the simulations, we discuss shortly the importance of the tuning parameter $k$, which in $k$-TJADE can be chosen separately for each mode. In the following, $(k_1, \ldots, k_r)$-TJADE is used to refer to $k$-TJADE with the value $k_m$ for the tuning parameter in the $m$th mode, $m \in \{ 1, \ldots , r \}$. Based on Theorem \ref{theo:limiting}, the computationally cheapest but still asymptotically optimal choice is the largest multiplicity of the kurtosis means in the current mode. This means that we have generally three choices for each mode: we may use the full TJADE if the mode is short; we can use $k$-TJADE for some small $k$ if the mode is long; or we can choose not to separate the mode at all if it is not expected to contain any relevant information, as might be the case, e.g., for the color dimension of a sample of images.

\subsection{Finite-sample efficiency, setting 1}\label{sub:perf1}
We begin by demonstrating that the finite sample performance of $k$-TJADE is in line with the asymptotic results given in Theorem \ref{theo:limiting}. In the first setting, we consider  simulated collections of i.i.d. matrices of size $3\times 3$, $ \textbf{Z} \coloneqq \{\textbf{Z}_j\}_{j\in \{1,\ldots, n\}}$. The components of every $\textbf{Z}_j$  are simulated independently from the   following distributions,
\begin{align*}
\textbf{Z}_j =
\begin{pmatrix}
\texttt{E} & \texttt{C} & \texttt{U} \\
\texttt{C} & \texttt{U} & \texttt{E} \\
\texttt{U} & \texttt{E} & \texttt{N} 
\end{pmatrix},
\end{align*}
where \texttt{C},  \texttt{E}, \texttt{U} and \texttt{N} denote independent replicates from the $\chi^2_1$, standard exponential, uniform and normal distribution, respectively, all scaled to have zero mean and unit variance.

In this simulation setting, none of the theoretical row or column kurtosis means are zero. However, the theoretical mean kurtoses are the same for the first two rows and the first two columns. Thus, the preferable $(k_1,k_2)$-TJADE procedure here is 22-TJADE, $ k_1 = k_2 = 2 $. Note that in this setting the requirements of TFOBI are not fulfilled.

When the observations are vectorized, the length nine vectors contain a single normally distributed component, two $ \chi^2 $-distributed elements, three elements from the exponential distribution and three elements from the uniform distribution. The vector now contains one element (the normally distributed component) with theoretical kurtosis 0, making vectorial ICA viable. The most natural $k$-VJADE procedure here is 3-VJADE. The assumptions of VFOBI are violated.

The simulation was performed using 13 different sample sizes, $n_i =2^{i-1}\cdot 1000$, $i\in\{1,\ldots13\}$ with $\ell=2000$ repetitions per sample size.
To evaluate the equivariance properties and the effect of non-orthogonal mixing,
we mixed the observations at each repetition using (i) identity mixing: $\boldsymbol{\Omega}_1=\boldsymbol{\Omega}_2 = \boldsymbol{I}_p$, (ii) orthogonal mixing: \red{$\boldsymbol{\Omega}_1=\boldsymbol{U}_1$ and $\boldsymbol{\Omega}_2=\boldsymbol{U}_2$, where $\boldsymbol{U}_1$ and $\boldsymbol{U}_2$} are random orthogonal matrices uniformly sampled at each repetition with respect to the Haar measure and (iii) normal mixing, where $\boldsymbol{\Omega}_1$ and $\boldsymbol{\Omega}_2$ were filled at each repetition with random elements from $N(0,1)$.

To evaluate the effect of the mixing matrix and the limiting distributions, we present, in Figure~\ref{fig:set1}, the transformed MD values for 3-VJADE, VJADE, 22-TJADE and TJADE and the corresponding theoretical limit values for the cases available. \red{As our theory extends only the orthogonal mixing, the theoretical limit value of TJADE is missing from the third panel.} The sample sizes $n_{i}$, $i \geq 8$, are omitted from Figure \ref{fig:set1}, since the sample size $64\cdot 10^3$  is large enough for the convergence of both the vectorial and the tensorial  methods.

Figure~\ref{fig:set1} clearly shows that for the vectorial methods, we have affine invariance and thus all of the three curves are identical. For the tensorial methods, we have only orthogonal invariance and hereby the curve corresponding to the normal mixing differs from the other two. \red{Moreover, despite the limiting value of TJADE being unknown for normal mixing, based on the finite-sample curves, it seems that normal mixing does not change the relative separation accuracy of the tensorial methods over orthogonal mixing. Only the overall separation difficulty level is affected (the curves are higher in the plots).} \red{Furthermore,} the benefit of applying the tensorial methods over the vectorized methods is impressive. Finally, Figure~\ref{fig:set1} also illustrates that all the methods converge to the theoretic values and that the $ k $-TJADE versions have the same limiting values as TJADE. 

\begin{figure}[t!]
	\begin{center}
		\includegraphics[width=1\textwidth]{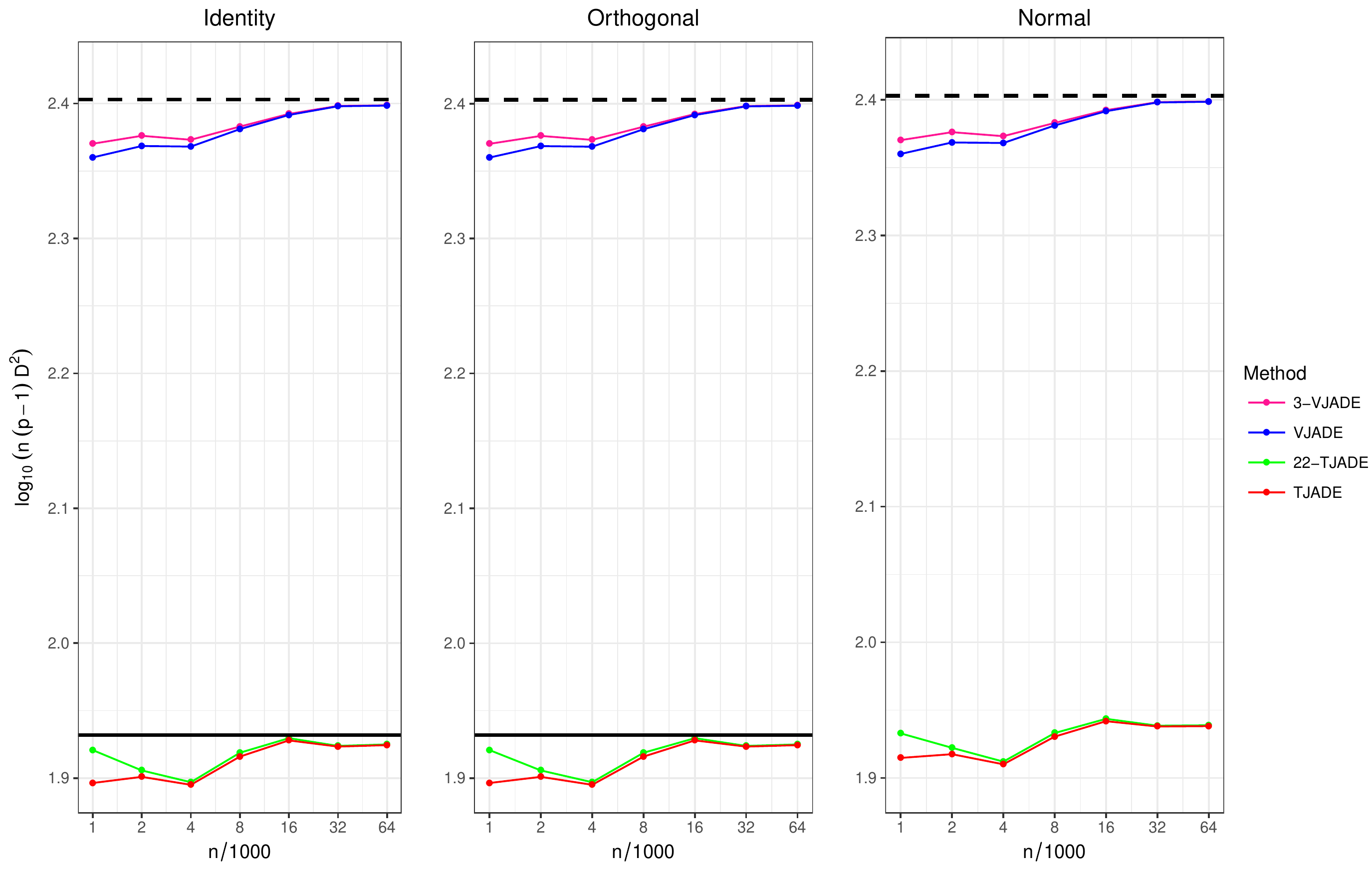}
	\end{center}
	\caption{Means of the transformed MD indices. The solid black line is the limiting value of TJADE (under orthogonal mixing) and the dashed black line is the limiting value of VJADE. \red{The solid line is missing from the third panel as the limiting value of TJADE is not known under normal mixing.} }
	\label{fig:set1}
\end{figure}

We next, under the same setting as above, present the results for seven additional procedures that violate some of the required assumptions: VFOBI, TFOBI, 1-VJADE, 2-VJADE, 11-TJADE, 12-TJADE and 21-TJADE. Note that TFOBI  is the initial step in the $k$-TJADE procedure and hereby the comparison between $k$-TJADE and TFOBI illustrates the added benefit of the additional rotation after the TFOBI-solution. Likewise, the same holds between VFOBI and $k$-VJADE. The resulting mean values of the transformed MD indices are presented in Figure \ref{fig:set1all}, where methods that have almost identical performance are presented using the same colors.

In Figure~\ref{fig:set1all}, TFOBI and VFOBI diverge at an exponential rate. Furthermore, the $(k_1,k_2)$-TJADE procedures that have either $k_1$ or $k_2$ less than the number of distinct kurtosis values, are not converging to anything reasonable, even at sample sizes greater than $4\cdot 10^6$. 
However, in this simulation setting, the $k$-VJADE procedure seems to allow slight deviations from the required assumptions. It seems that 2-VJADE converges to an asymptotic value that is at least close to that of VJADE, see Section \ref{sec:summary} for further discussion. 

To summarize this part of the simulation study: $k$-TJADE works as expected, when the theoretical conditions are met, and the convergence to the asymptotic value is relatively fast. Furthermore, even though $k$-TJADE is not affine invariant, its performance is better under all mixing scenarios, when compared to the  affine invariant  vectorial counterparts.
\begin{figure}[t!]
	\begin{center}
		\includegraphics[width=1\textwidth]{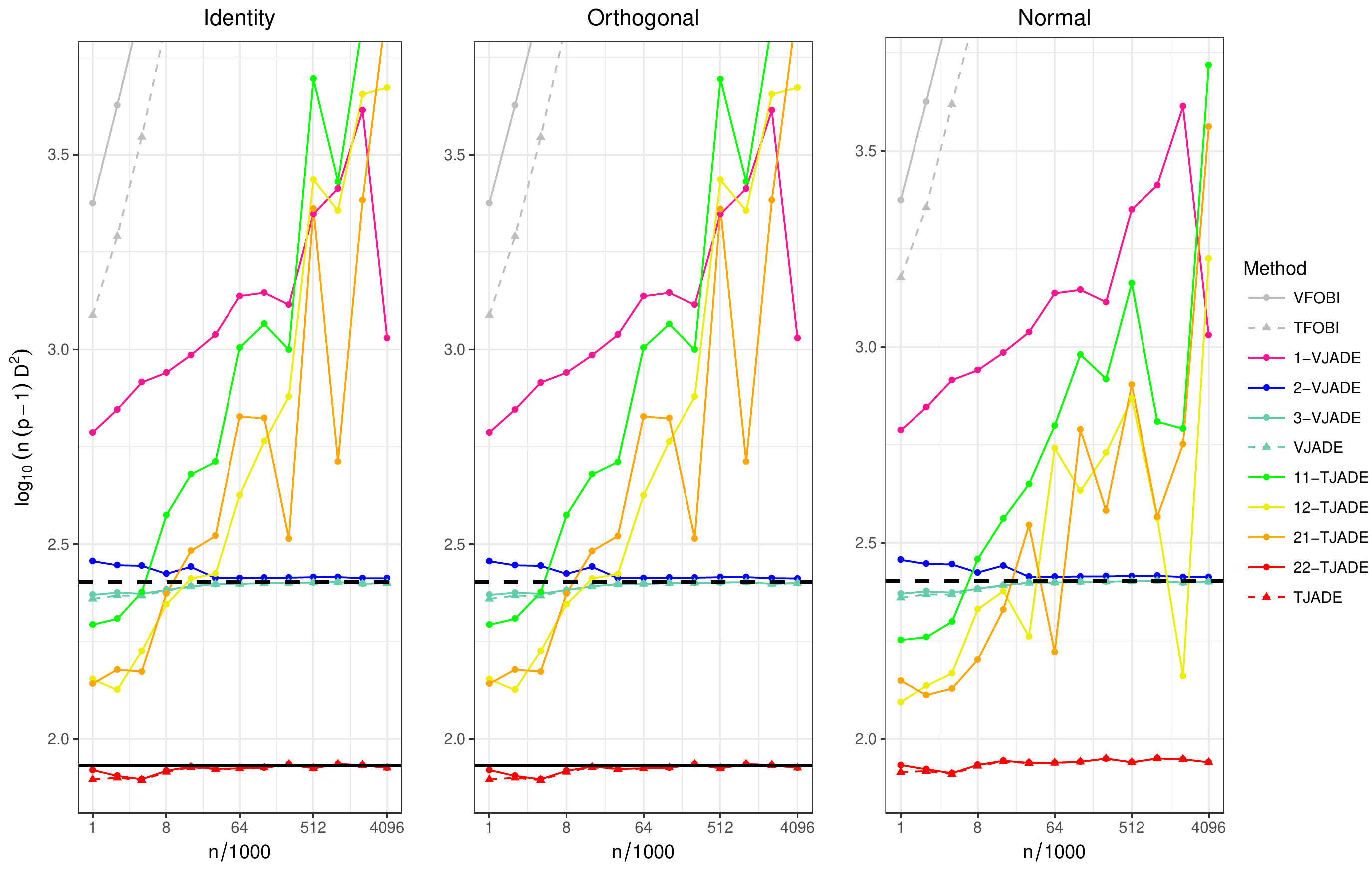}
	\end{center}
	\caption{Means of the transformed MD indices. The solid black line is the limiting value of TJADE (under orthogonal mixing) and the dashed black line is the limiting value of VJADE. \red{The solid line is missing from the third panel as the limiting value of TJADE is not known under normal mixing.}}
	\label{fig:set1all}
\end{figure}

\subsection{Finite-sample efficiency, setting 2}\label{sub:perf2}
In our second simulation, we illustrate unmixing under a tensorial setting. 
We consider  simulated collections of i.i.d. tensors of size $3\times 3 \times 4$, $ \ten{Z} \coloneqq \{\ten{Z}_j\}_{j\in \{1,\ldots, n\}}$. Let $\textbf{Z}^{(3)}_{jk}$ denote the $k$th 3-mode face of $\ten{Z}_{j}$. The components of every $\ten{Z}_j$  are then simulated independently from the   following distributions, 
\begin{align*}
\textbf{Z}^{(3)}_{j1} =
\begin{pmatrix}
\texttt{E} & \texttt{N} & \texttt{N} \\
\texttt{N} & \texttt{U} & \texttt{N} \\
\texttt{N} & \texttt{N} & \texttt{E} 
\end{pmatrix}, \quad
\textbf{Z}^{(3)}_{j2} =
\begin{pmatrix}
\texttt{E} & \texttt{N} & \texttt{N} \\
\texttt{N} & \texttt{U} & \texttt{N} \\
\texttt{N} & \texttt{N} & \texttt{E} 
\end{pmatrix}, \quad
\textbf{Z}^{(3)}_{j3}  =
\begin{pmatrix}
\texttt{E} & \texttt{N} & \texttt{N} \\
\texttt{N} & \texttt{U} & \texttt{N} \\
\texttt{N} & \texttt{N} & \texttt{E} 
\end{pmatrix}, \quad
\textbf{Z}^{(3)}_{j4}  =
\begin{pmatrix}
\texttt{N} & \texttt{U} & \texttt{E} \\
\texttt{E} & \texttt{E} & \texttt{E} \\
\texttt{E} & \texttt{E} & \texttt{N} 
\end{pmatrix},
\end{align*}
where the different distributions are denoted as in Section \ref{sub:perf1}.

All of the mean kurtoses over the different tensor faces are nonzero and none of the theoretical mean kurtoses are the same for the 1-mode faces. Moreover, two of them are the same for the 2-mode faces and three of them are the same for the 3-mode faces. Hereby, the preferable $(k_1,k_2,k_3)$-TJADE here is $123$-TJADE, $ k_1 = 1, k_2 = 2, k_3 = 3 $.
The vectorized versions of the observations contain several normal components and thus the assumptions for the vectorial methods are not satisfied here.

The simulation was performed using 11 different sample sizes, $n_i =2^{i-1}\cdot 1000$, $i\in\{1,\ldots11\}$ and the simulation was repeated 2000 times for each  sample size.
We considered the same three mixing scenarios as in Section \ref{sub:perf1} and generated the mixing matrices in the same way, with the distinction that here we have three mixing matrices instead of two.  We performed the unmixing using VFOBI, TFOBI, 1-VJADE, 2-VJADE, VJADE, TJADE and 11 different versions of $k$-TJADE. The resulting mean values of the transformed MD index, $\ell^{-1}\sum_{j=1}^{\ell} n_i (\rho - 1) D(\hat{\textbf{G}}_j)^2$, where $\ell=2000$ and $\rho =  36$,  are presented in Figure~\ref{fig:set2}. The orthogonal mixing is omitted from  Figure \ref{fig:set2}, since the tensorial methods are orthogonally invariant and the vectorial methods are affine invariant. The performance curves under the orthogonal mixing would be identical to those under the identity mixing, similarly as in Section \ref{sub:perf1}.

The $k$-TJADE performs as expected for the values of $k_1,k_2,k_3$ that satisfy the assumptions and  the convergence towards the theoretical limit is considerably fast. The vectorial methods fail completely in this example. Interestingly, $k$-TJADE has relatively nice performance when the elementwise deviation between  $(k_1,k_2,k_3)$  and (1,2,3) is not too large. See Section \ref{sec:summary} for further discussion. 

\begin{figure}[t!]
	\begin{center}
		\includegraphics[width=1\textwidth]{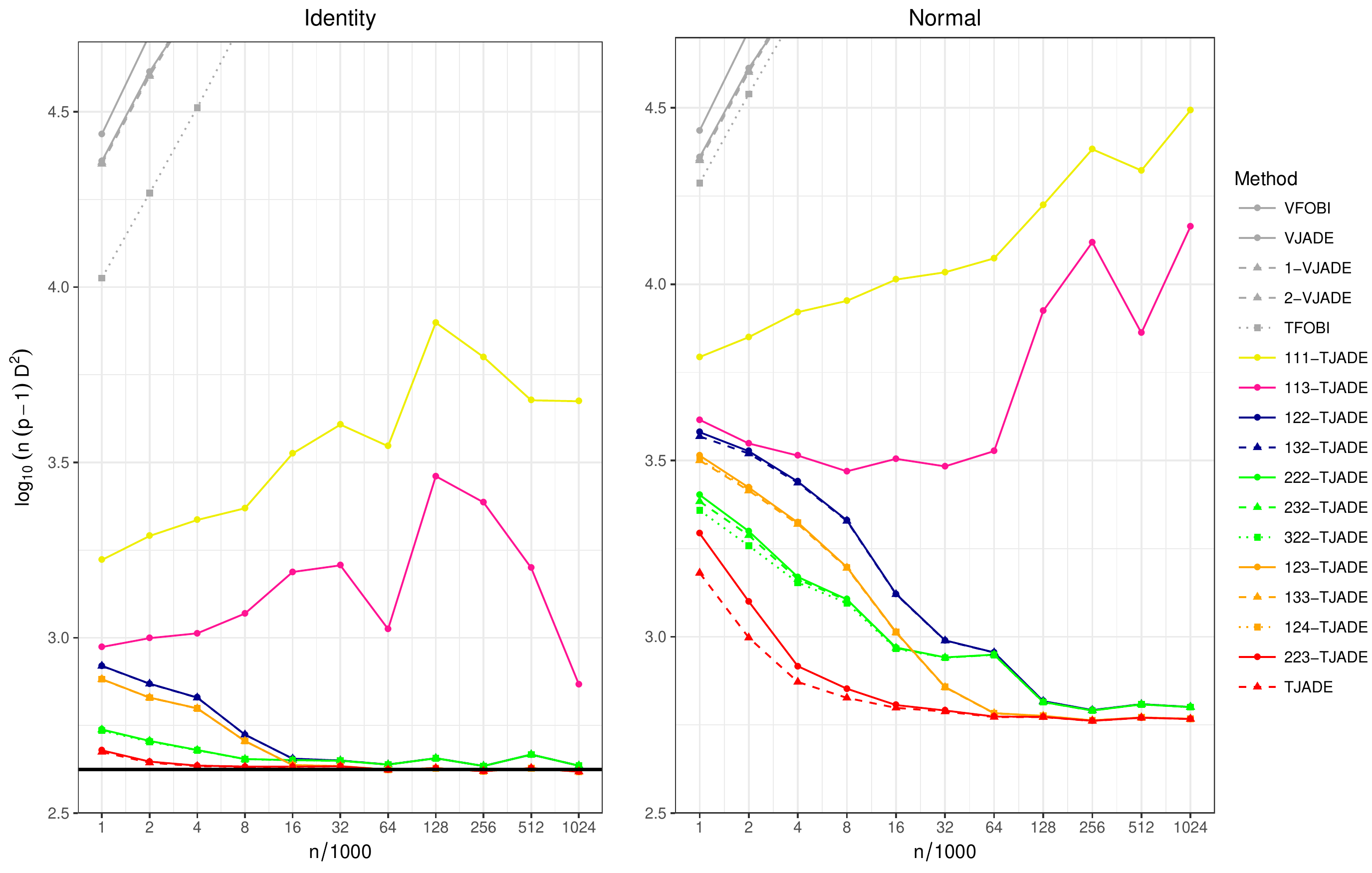}
	\end{center}
	\caption{Means of the transformed MD indices. The solid black line is the limiting value of TJADE (under orthogonal mixing).  }
	\label{fig:set2}
\end{figure}

\joni{
\subsection{Finite-sample efficiency, setting 3}\label{sub:perf3}
In the third simulation setting, we have a collection of i.i.d. matrices  $ \textbf{Z} \coloneqq \{\textbf{Z}_j\}_{j\in \{1,\ldots, n\}}$, with components simulated independently from the following distributions,
\begin{align*}
\textbf{Z}_j =
\begin{pmatrix}
\texttt{G}(0.9999) & \texttt{G}(1) & \texttt{G}(0.9) \\
\texttt{G}(0.9) & \texttt{G}(0.9998) & \texttt{G}(1) \\
\texttt{G}(0.9) & \texttt{G}(1) & \texttt{G}(1) 
\end{pmatrix},
\end{align*}
where $\texttt{G}(\alpha)$ denotes the gamma distribution, with shape parameter $\alpha$ and rate parameter 1, scaled to have zero mean and unit variance. The theoretical row kurtosis means are approximately,
\begin{align*}
\begin{pmatrix}
6.222422 &6.222622 &6.222222
\end{pmatrix},
\end{align*}
and the theoretical column kurtosis means are approximately,
\begin{align*}
\begin{pmatrix}
6.444644 &6.000400 &6.222222
\end{pmatrix}. 
\end{align*}

Hereby, in this simulation setting, the row kurtosis means are very close to each other. They start to differ at the fourth decimal. Also the column kurtosis means are quite close to each other.

The simulation study was conducted exactly as in Section \ref{sub:perf1}, with the distinction that the largest sample size of Section \ref{sub:perf1} was omitted here. Furthermore, we had more versions of $k$-TJADE in this study. The results are displayed in Figure \ref{fig:set3}.

Theoretically, the preferable $(k_1,k_2)$-TJADE would be 11-TJADE, $k_1 = 1, k_2 = 1$. However, with the current sample sizes, 11-TJADE, 12-TJADE and 13-TJADE do not seem to converge, see Figure \ref{fig:set3}. Similarly, 21-TJADE and 31-TJADE exhibit erratic behavior with small sample sizes. Yet, with sample sizes over one million, the corresponding procedures are very close to the theoretical limiting value, given as the solid black line in Figure \ref{fig:set3}. It seems that with sample sizes of magnitude $10^6$, differences in column kurtosis means that are of magnitude $10^{-1}$, are not a problem for the procedures. Conversely, the current sample sizes are not large enough to compensate the small differences in the row kurtosis means. 

The curves corresponding 22-TJADE, 23-TJADE, 32-TJADE and TJADE behave nicely in Figure \ref{fig:set3}, even with relatively small sample sizes. Especially, the \red{good} behavior of 22-TJADE is a little surprising. The corresponding procedure assumes that at most two of the column kurtosis means and at most two of the row kurtosis means are equal. We wish to emphasize that the assumptions required for the $k$-TJADE procedures are only sufficient conditions and it can be that in some special cases, e.g., under some special distributional structures, some of the procedures behave \red{well}, even though the sufficient assumptions are not satisfied. 

In Figure \ref{fig:set3}, the curves that correspond to VFOBI and TFOBI are again only visible in the top left corner. This is unsurprising since the required assumptions are not satisfied for VFOBI, and TFOBI cannot handle the close kurtosis means, even though they are distinct on the population level. With the current sample sizes, 1-VJADE does not seem to converge to the theoretical limiting value, given as the dashed black line. The vectorial 2-VJADE, 3-VJADE and VJADE seem to function quite well in this setup. This could be explained by the fact that the vectorial methods have assumptions directly related to the theoretical kurtoses of the components, whereas the tensorial methods have assumptions related to means of the kurtoses. It is easier to detect small kurtosis differences, when compared to detecting differences in their means. 
}

\begin{figure}[t!]
	\begin{center}
		\includegraphics[width=1\textwidth]{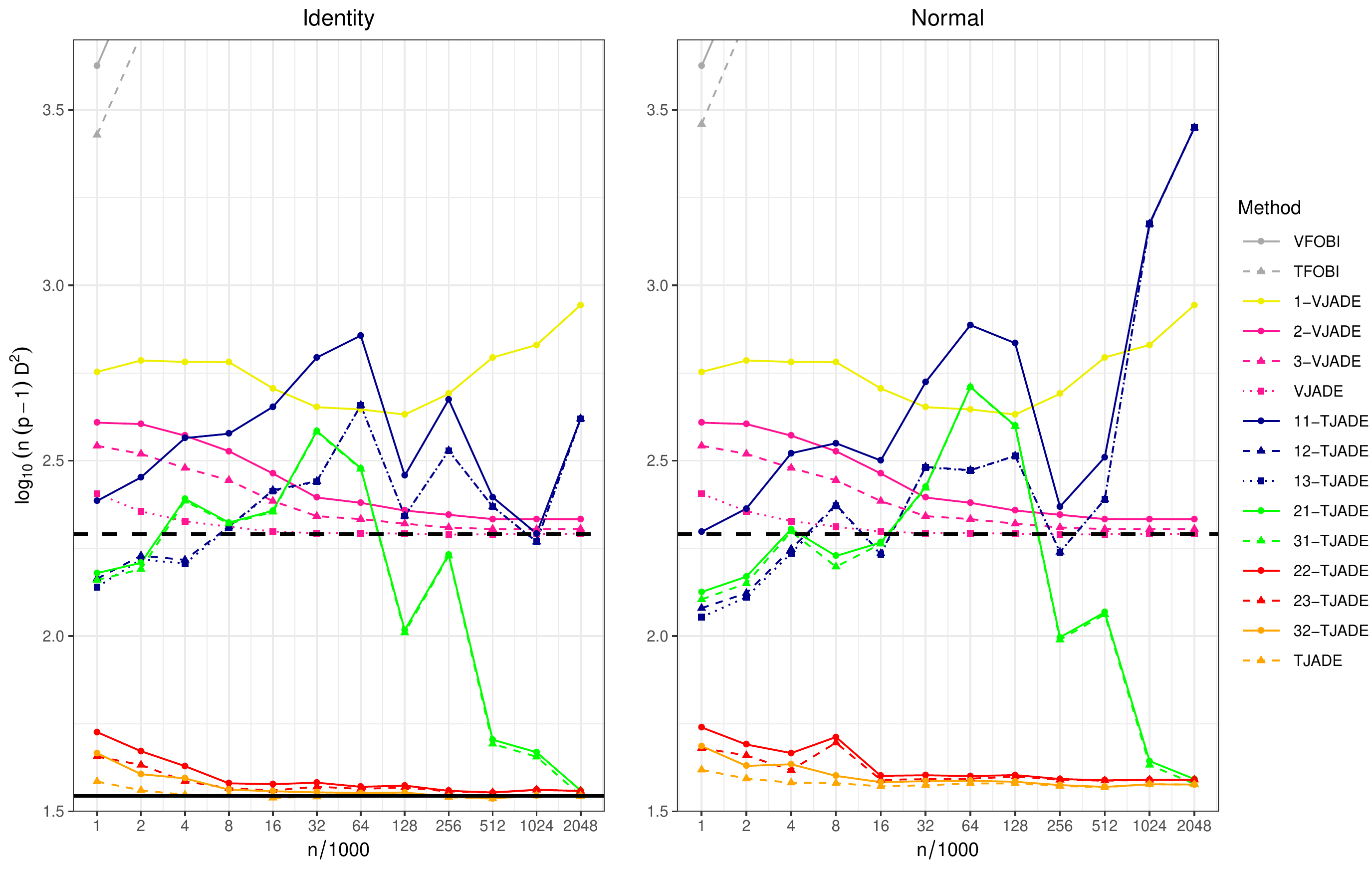}
	\end{center}
	\caption{Means of the transformed MD indices. The solid black line is the limiting value of TJADE (under orthogonal mixing) and the dashed black line is the limiting value of VJADE.}
	\label{fig:set3}
\end{figure}

\subsection{Timing comparison}

The results from Sections \ref{sub:perf1} and \ref{sub:perf2} illustrate that the performances of TJADE and suitably chosen versions of $ k $-TJADE are very similar. Next, we quantify the significant improvement in computational speed that $k$-TJADE provides when compared to TJADE.  In the timing comparison,  we consider a simulated collection of i.i.d. matrices of size $3\times q$, $ \textbf{Z} \coloneqq \{\textbf{Z}_j\}_{j\in \{1,\ldots, n\}}$, such that components of every $\textbf{Z}_j$  are simulated independently from the   following distributions,
\[ 
\textbf{Z}_j =
\begin{pmatrix}
\chi^2_1 & \chi^2_4 & \cdots & \chi^2_{3(q-1)+1}\\
\chi^2_2 & \chi^2_5 & \cdots & \chi^2_{3(q-1)+2}\\
\chi^2_3 & \chi^2_6 & \cdots & \chi^2_{3(q-1)+3}
\end{pmatrix},
\]
where $\chi^2_\nu$ denotes the $ \chi^2 $-squared distribution with $ \nu $ degrees of freedom, and the width $ q $ of the matrix is the varying parameter in this simulation setting.  We used parameter values $ q = 5, 10, 15, 20, \ldots , 50 $ and the sample size $n=1000$. We considered the same procedures as in Section \ref{sub:perf1}: VFOBI, TFOBI, 1-VJADE, 2-VJADE, 3-VJADE, VJADE, 11-TJADE, 12-TJADE, 21-TJADE, 22-TJADE, TJADE and recorded the mean running times over a total of 5 iterations. The time it took R to vectorize the tensors was also considered as a part of the vectorized procedures. However, the time the vectorizing took, was negligible.  We used two alternative stopping criteria for the methods that involve joint diagonalization, that is, for all the methods except for TFOBI and VFOBI.
A single iteration was stopped, if either the converge tolerance of the Jacobi rotation based joint diagonalization was less than $10^{-6}$ or if the required tolerance was not satisfied after 100 iterations.

The average running times in minutes as a function of the dimension $ q $ are presented in Figure \ref{fig:time} and methods that have almost identical computing times, are presented using the same colors. Figure \ref{fig:time} clearly illustrates the superior computation speed of $k$-TJADE, when compared to either TJADE or any of the vectorized counterparts. The timing comparison was conducted on Ubuntu 16.04.4 LTS  with  Intel\textsuperscript{\textregistered} Xeon\textsuperscript{\textregistered} CPU E3-1230 v5 with 3.40GHz and 64GB.

\begin{figure}[t!]
	\begin{center}
		\includegraphics[width=1\textwidth]{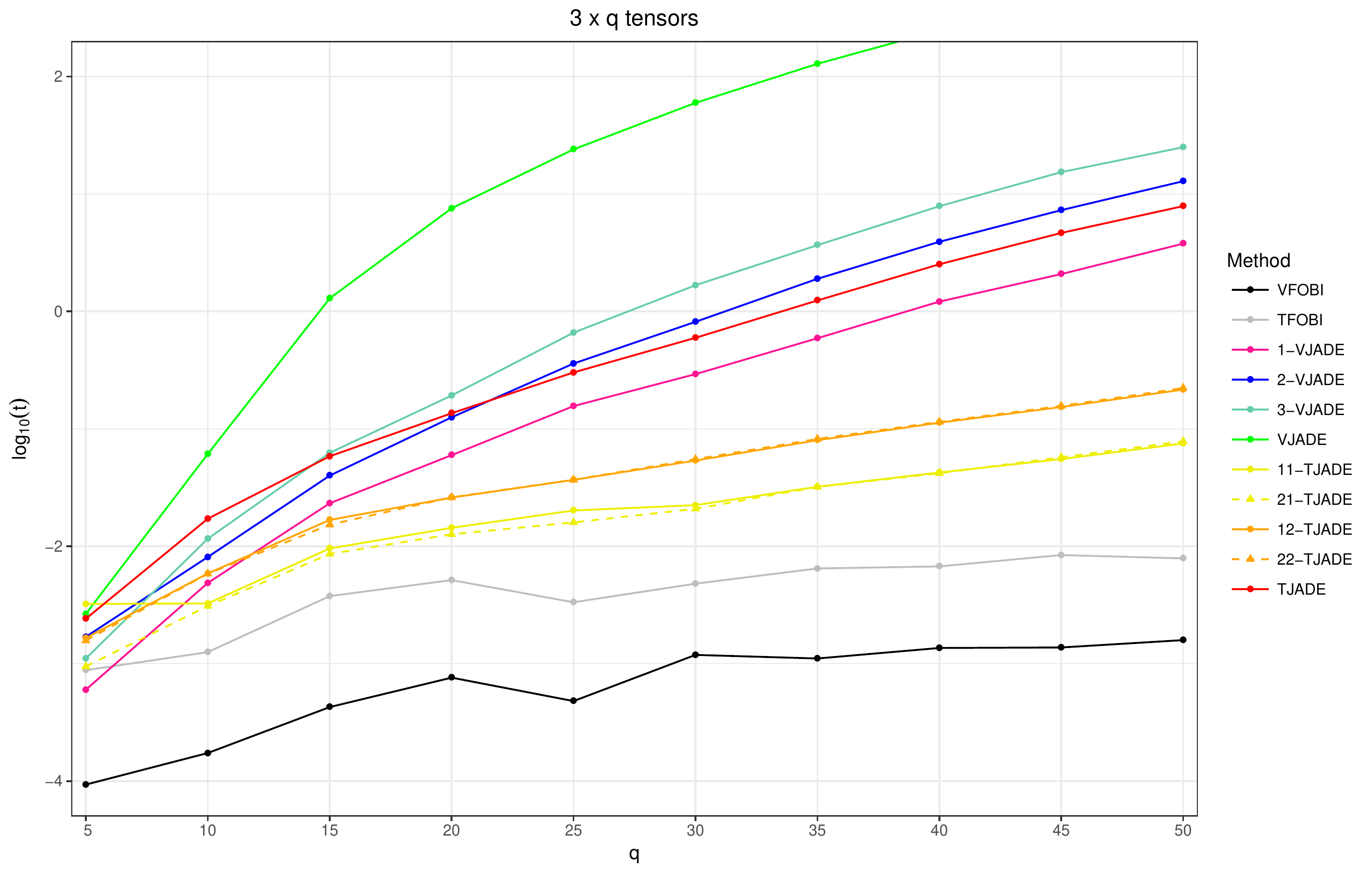}
	\end{center}
	\caption{The logarithms of the mean computation times (in minutes) as a function of the dimension $q$. }
	\label{fig:time}
\end{figure}

\subsection{Video example}

We applied $ k $-TJADE to the WaterSurface surveillance video \citep{li2004statistical} that is viewable at \url{http://pages.cs.wisc.edu/~jiaxu/projects/gosus/supplement/} and available to download as an .Rdata file at \url{https://wis.kuleuven.be/stat/robust/Programs/DO/do-video-data-rdata}. The video has already been used as an example for blind source separation \citep{VirtaNordhausen:2017tnss,virta2017blind}. Each frame of the video is of size $ h \times w \times 3 $ with the height $ h = 128 $, width $ w = 160 $ and a three-variate color channel (RGB). The total video consists of $ 633 $ such frames, making our data a sample of size $ n = 633 $ of random third order tensors in $\mathcal{R}^{h \times w \times 3}$. The data constituting a single continuous surveillance video, the observations are naturally not independent and the assumptions of tensorial independent component analysis are not fully satisfied. However, ICA is known to be robust against deviations from the independence assumption and applying it to sets of dependent data with success is common practice. We thus expect $k$-TJADE to successfully extract components of interest from our current data.

\begin{figure}[t!]
	\begin{center}
		\includegraphics[width=1\textwidth]{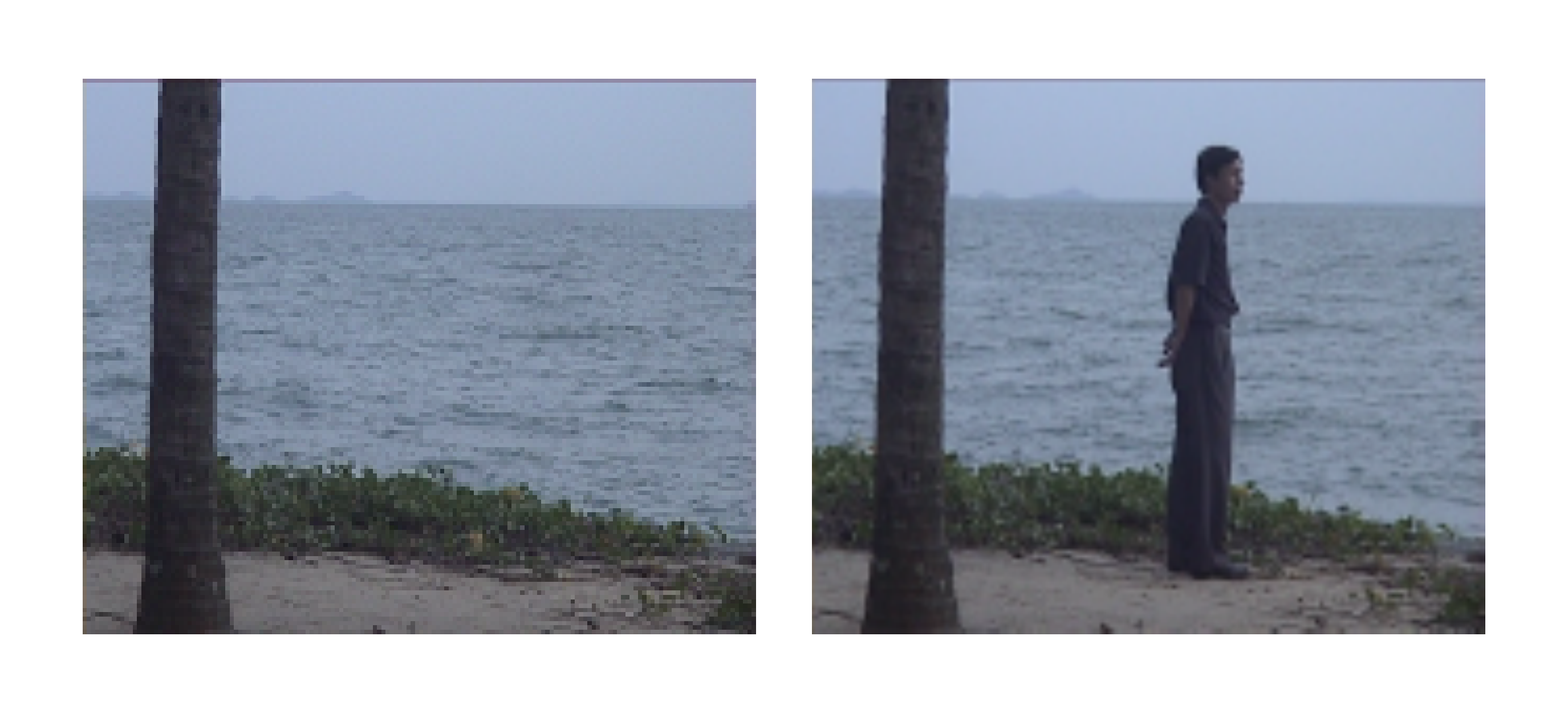}
	\end{center}
	\caption{Frames 50 and 550 of the surveillance video data.}
	\label{fig:frames}
\end{figure}

The video shows a beach scene with little to no motion until frame 480, when a man enters the scene from the left, staying in the picture for the remainder of the video. Figure \ref{fig:frames} shows frames 50 and 550 of the video, illustrating moments before and after the man enters the scene. Our objective with the surveillance video is to find low-dimensional components that allow us to pinpoint the most obvious change point in the video, namely, the man's entrance. As change points are most likely captured in the data as outliers of some form, it is natural to seek the component of interest among those with high kurtosis values.

\begin{figure}[t!]
	\begin{center}
		\includegraphics[width=1\textwidth]{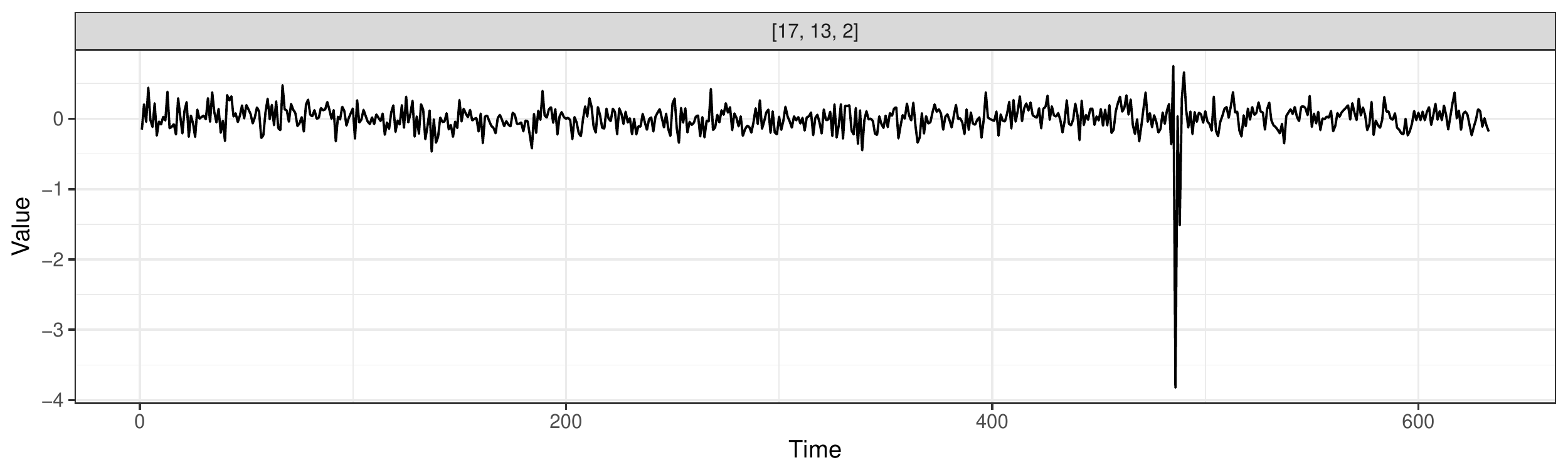}
	\end{center}
	\caption{The source with the highest absolute kurtosis found from the video data using $(1, 1, 0)$-TJADE. The plot caption refers to the indices of the source in the extracted matrix $\textbf{Z}$.}
	\label{fig:sources}
\end{figure}

We proceed as follows: we run $ k $-TJADE on the data with different choices of the tuning parameters, find the component with the highest absolute kurtosis, and plot its time course to visually assess whether it captures the change point. The component found with $ (1, 1, 0) $-TJADE is shown in Figure \ref{fig:sources}, where $ k_3 = 0 $ means that we do not unmix the supposedly uninformative color dimension at all. The time series is instantly seen to capture the desired time point as the spike coincides with the first frames the man spends in the scene. The running time of the method was 39 minutes on a Windows server with Intel\textsuperscript{\textregistered} Xeon\textsuperscript{\textregistered} CPU R5 2440 with 2.40GHz and 64GB. Applying $ (2, 2, 0) $-TJADE gave almost identical results with the increased running time of one hour and 54 minutes. However, the original TJADE proved to be very slow. The running time of the algorithm was over five days. Concluding, the example shows that $ k $-TJADE can be used to reliably extract information from data where TJADE can not be applied due to its extremely high computational cost. In several real world applications, e.g. in crime scene investigation, waiting for days is not an option.

\subsection{\red{Estimation of the maximal kurtosis multiplicity $ \nu $}}

\red{In the previous video example, applying $ (1, 1, 0) $-TJADE already gave satisfactory results in the sense that it allowed us to find a latent component identifying the point of interest in time. However, in some cases, using the preferable value $ k = \nu $ for the tuning parameter might be justified. This will guarantee the successful estimation of all components, while simultaneously keeping computation times and the amount of noise in the estimation minimal. The results of Theorem~\ref{theo:limiting} could possibly be used to formulate an asymptotic hypothesis test for the null hypothesis that $ \nu \leq \nu_0 $ for some given $ \nu_0 $ in a given mode. Namely, if the null hypothesis holds, then the limiting distributions of the unmixing estimates $\hat{\boldsymbol{\Gamma}}{}^{\nu_0}, \hat{\boldsymbol{\Gamma}}{}^{\nu_0 + 1}, \ldots , \hat{\boldsymbol{\Gamma}}{}^{p}$ are identical, allowing us to pin-point the true value of $ \nu $. In the following, we will pursue this idea from an empirical perspective and develop a heuristic procedure for the estimation of $ \nu $.}

\red{Consider a fixed mode of size $ p $ of a sample of data $ \ten{X}_1, \ldots , \ten{X}_n $ from the tensorial IC model in Eq. \eqref{eq:tensor_model}  and let $\hat{\boldsymbol{\Gamma}}{}^{1}, \hat{\boldsymbol{\Gamma}}{}^{2}, \ldots , \hat{\boldsymbol{\Gamma}}{}^{p}$ be the unmixing matrices estimated from the corresponding data, respectively, to $ k $-TJADE with the value of $ k \in \{ 1, \ldots p \} $. Letting $ \nu $ be the true maximal kurtosis multiplicity in the data, we thus expect $ \hat{\boldsymbol{\Gamma}}{}^{\nu} $ to separate the data well but $ \hat{\boldsymbol{\Gamma}}{}^{\nu - 1} $ to leave some parts still unmixed (within the multiplicities). Analogously, we expect $ m_{\nu-1} := D( \hat{\boldsymbol{\Gamma}}{}^{\nu - 1} (\hat{\boldsymbol{\Gamma}}{}^{\nu})^{-1}) $, where $ D(\cdot) $ is the MD index in Eq. \eqref{eq:mdindex}, to be large as the two matrices should have structures differing beyond row permutations, scaling and sign changes (in the extreme case where $ \hat{\boldsymbol{\Gamma}}{}^{\nu - 1} = \hat{\boldsymbol{\Gamma}}{}^{\nu} $ we have $ m_{\nu - 1} = 0 $). On the other hand, we expect all the following sequential MD indices, $ m_{\nu + \ell} $, $ \ell \in \{ 0, \ldots , p - \nu - 1 \} $, to be small as all values of $ k \in \{ \nu, \ldots, p \} $ are sufficient for the separation, and yield asymptotically equal solutions.
}

\red{The true value of $ \nu $ can now be located by plotting $ m_k $ versus $ k \in \{ 1, \ldots , p - 1 \}$ and finding the value, starting from which the curve stays roughly constant. However, since using a single sequential MD index per value of $ k $ might leave us with a curve that is not smooth enough to accurately distinguish the true $ \nu $, we compute, for each $ k \in \{ 1, \ldots , p - 1 \} $, the average of the forward sequential MD indices:
\[ 
m^*_k := \frac{1}{p - k} \sum_{\ell = 1}^{p - k} D( \hat{\boldsymbol{\Gamma}}{}^{k} (\hat{\boldsymbol{\Gamma}}{}^{k + \ell})^{-1}).	
\]
The plot of $ m^*_k $ versus $ k $ can be used to find the true value of $ \nu $ similarly as with the plot of $ m_k $ versus $ k $, and has the interpretation of quantifying the efficiency loss one encounters when using a particular value of $ k $, relative to the choice $ k = p $. Note also that the choice of $ k $ can be done individually for each mode in the sense that the value of $ m^*_k $ in $ (k, k_0) $-TJADE is invariant to the choice of $ k_0 $, and similarly for higher order data.

\begin{figure}[t!]
	\begin{center}
		\includegraphics[width=1\textwidth]{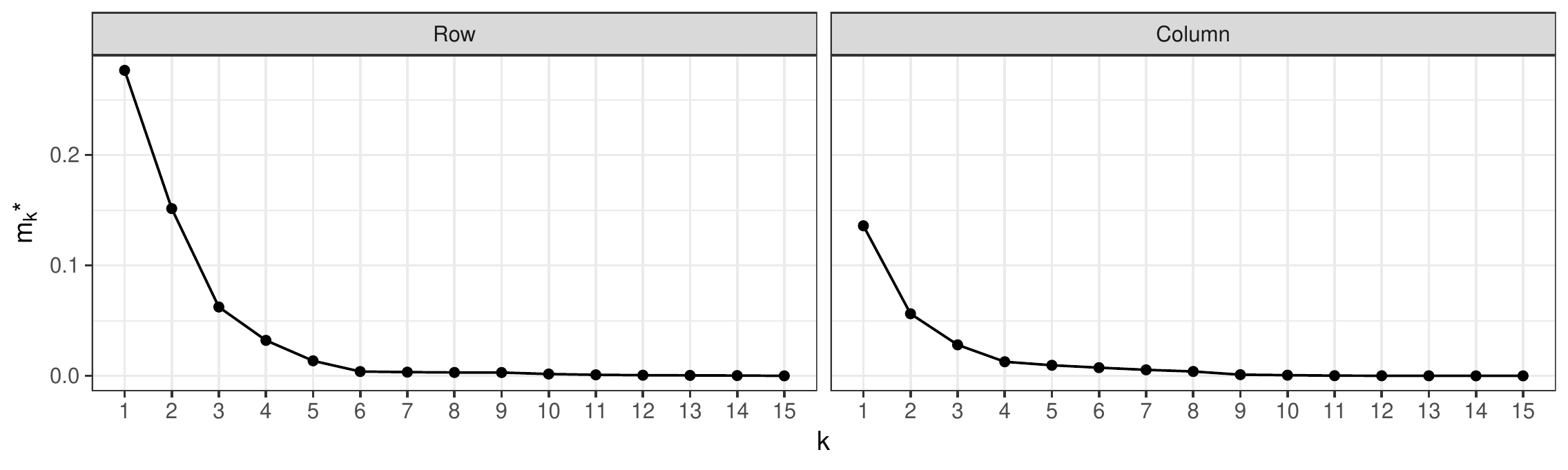}
	\end{center}
	\caption{The plots of the sequential MD index means $ m^*_k $ versus the tuning parameter value $ k $. The left plot is for the rows and the right plot for the columns of the digit data.}
	\label{fig:digit_mdi}
\end{figure}

To illustrate the proposed tool, we consider the hand-written digit data set from the R-package \textit{ElemStatLearn} \citep{RElemStatLearn}. The data set consists of 7291 images of hand-written numerals 0--9 digitized to $ 16 \times 16 $ matrices, with each element in $ [-1, 1] $, describing the grayscale intensity of the corresponding pixel. The underlying objective of the data set is to build a classifier that can correctly identify the digits that each of the matrices represent. For simplicity, we work on a subset of 400 randomly chosen images of the digits 1 and 7. These two digits were chosen as their visual similarity makes the classification task more difficult. In this kind of context, ICA is commonly used as a preprocessing step to reduce the data into a low-dimensional subspace, which hopefully contains all the classification information, simplifying the task for the subsequent classifier. In this spirit, we apply the proposed estimation strategy for $ \nu $ to the sample and obtain the two plots shown in Figure \ref{fig:digit_mdi}. The plots should be interpreted similarly as the scree plot in PCA, where the aim is to find an ``elbow'' where the slope changes quickly. The two curves seem to imply that the correct parameter for $ k $-TJADE is roughly $ (4, 3) $ . In Figure \ref{fig:digit_plot} we plot the components with the indices $ (16, 16)  $ and $ (16, 15) $ of the $ (4, 3) $-TJADE solution. These components were chosen as $ k $-TJADE orders the rows/columns in descending order according to their mean kurtoses, and low kurtosis is often a sign of bimodality, making it natural to search for components with classifying ability in the lower-right corner of $ \textbf{Z} $. Indeed, the chosen components reveal a clear separation of the digits. All ones are concentrated roughly in one spherical cluster, with a single outlying seven inside. The bulk of the sevens lies around the cluster of ones in a curved manner and the location of an individual seven is based on the angle of its stem. The analysis could further be continued from here by fitting a classifier, for example, a support vector machine, to the obtained components.

\begin{figure}[t!]
	\begin{center}
		\includegraphics[width=1\textwidth]{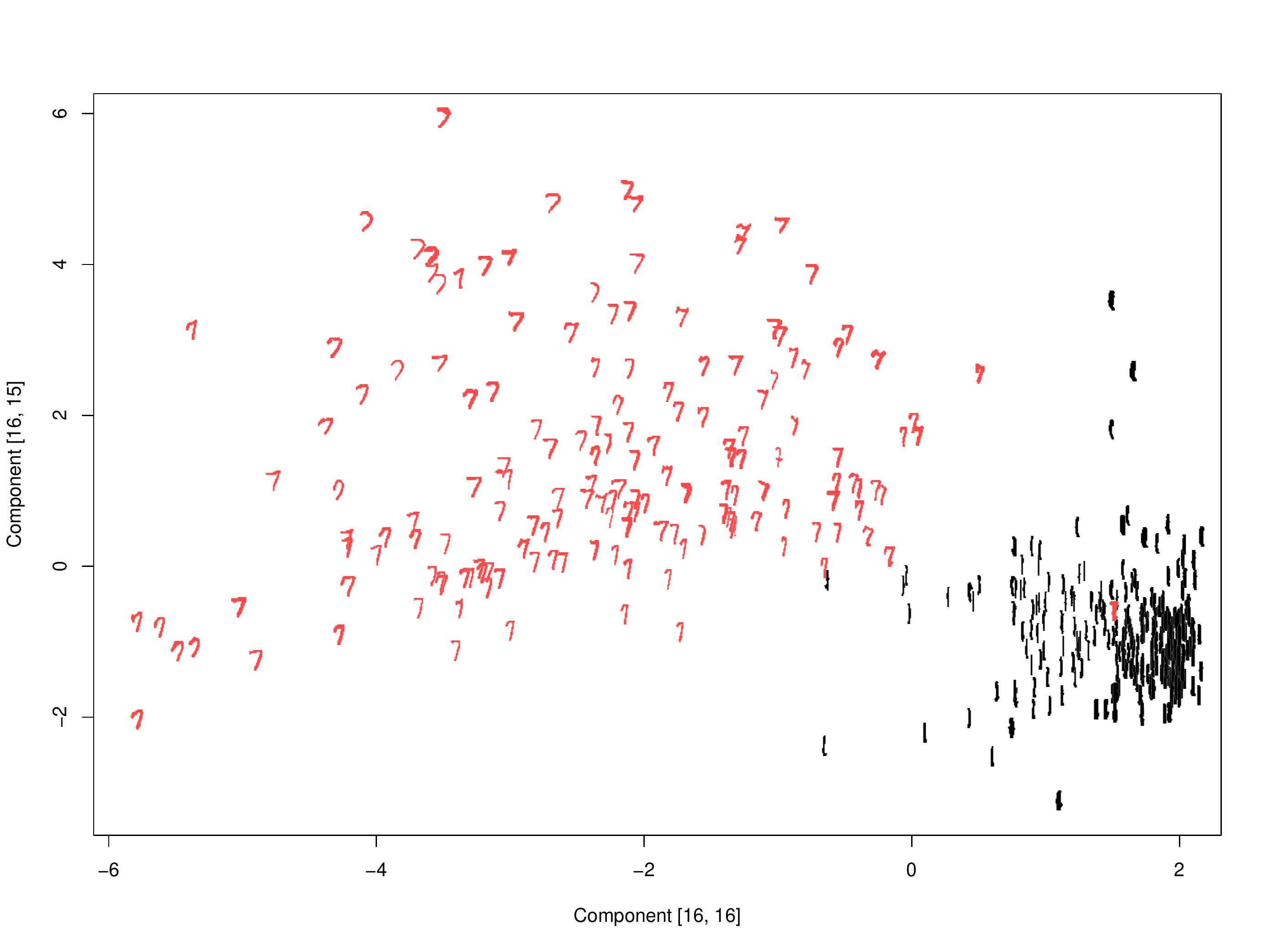}
	\end{center}
	\caption{The scatter plot of two particular low-kurtosis independent components found by $ (4, 3) $-TJADE from the digit data.}
	\label{fig:digit_plot}
\end{figure}

We conclude with two remarks. First, experimenting reveals (not shown here) that also, e.g., $ (2, 2) $-TJADE finds similar separation of the two groups as is shown in Figure \ref{fig:digit_plot}. There are two explanations for this: On one hand, locating the elbows in the scree plots in Figure \ref{fig:digit_mdi} is visually difficult for a sample size this low, and thus it could be that $ (2, 2) $ is the actual optimal choice. On the other hand, whatever the true value of $ \nu $ is, using $ k $ larger than that is sure to guarantee that \textit{all} independent components are estimated (asymptotically) correctly. However, it could be that the group separation structure is such a strong feature of the data that it can be found also with a sub-optimal choice of $ k $, and that using the optimal value of $ k $ just guarantees that also the $ 254 $ other (possibly less interesting) components are recovered successfully.

Second, the suggested procedure is somewhat excessive, as it requires the computation of the $ k $-TJADE solution for a wide range of values of $ k $. In addition, after having computed the solutions, one option would be to simply continue the analysis using the solution with the largest value of $ k $ and ignore the estimation altogether. However, when dealing with very low sample sizes, it could happen that using a too large $ k $, although asymptotically sufficient for the estimation of all components, induces noise in the estimation. Still, the suggested procedure could prove useful in studies where a pilot/training data set is used to determine the optimal value of $ k $, which is then used for later data sets, saving computation time for all other data sets at the expense of the first. Finally, one option to save computation time would be to compute the values of $ m_k^* $ sequentially, increasing $ k $ one-by-one, and stopping when a significant drop in the plot is visible.  
}

\section{Discussion}\label{sec:summary}

\red{We proposed a sped-up version of TJADE, the method with the lowest limiting variance among the currently studied methods of tensorial independent component analysis. Under easily interpretable additional assumptions, the extension, $ k $-TJADE, achieves the same limiting variance to TJADE, while simultaneously exhibiting significantly lower computational cost.} A large part of this efficiency is preserved also for samples of finite size.

An interesting future research question is to derive the theoretic behavior of $k$-TJADE when Assumption \ref{assu:FOBI}$ (v) $ is violated. Based on our simulations, even when the value of $ k $ is chosen to be too small or too large, $k$-TJADE and $k$-JADE can still work. This can be seen as a safety net for the users of $k$-TJADE. The simulations suggest that the performance deteriorates the further down one goes from the optimal $k$.

\section*{Acknowledgements}
The work of Joni Virta was supported by the Academy of Finland (grant 321883). The work of Niko Lietz\'{e}n was supported by the Emil Aaltonen Foundation (grant 170156 N). The work of Klaus Nordhausen was supported by Austrian Science Fund (FWF) Grant number P31881‐N32 and CRoNoS COST Action IC1408.  All authors acknowledge the computational resources provided by the Aalto Science-IT project.

\red{The authors wish to express their gratitude to the two anonymous referees, whose insightful comments greatly helped improving the quality of the manuscript.}

\appendix

\section{Proofs}\label{sec:proofs}


We denote the sequence of i.i.d. observations as $\textbf{X}_N  = \{ \textbf{X}_1, \ldots, \textbf{X}_n \}$, such that e.g. $\hat{\boldsymbol{\Sigma}}_1[ \textbf{X}_N]$ denotes  the left sample covariance matrix  estimated from the sample $\textbf{X}_N$.

Before the proofs of the main results, we  establish three auxiliary lemmas.

\begin{lemma}
	\label{lemma:diag}
	The minimization problem   ${\arg\!\min}_{\textbf{V}  \in  \mathcal{U}} \left\{\tilde{g}\left(\textbf{V},\textbf{X}\right)\right\}$ is equivalent to the maximization problem  ${\arg\!\max}_{\textbf{V}  \in  \mathcal{U}} \left\{{g}\left(\textbf{V},\textbf{X}\right)\right\}$, where $\mathcal{U} = \{ \textbf{V} \in \mathcal{R}^{p\times p}   : \textbf{VV}\tprime = \textbf{I}_p  \} $,
	\begin{align*}
	\tilde{g}\left( \textbf{V},\textbf{X} \right) = \sum_{i,j=1}^p \left\| \mathrm{off} \left( \textbf{V}\tprime \textbf{C}^{ij}[\textbf{X}] \textbf{V} \right)  \right\|_\textnormal{F}^2 \quad \text{and}
	\quad 
	g\left( \textbf{V},\textbf{X}  \right) = \sum_{i,j=1}^p \left\| \mathrm{diag} \left( \textbf{V}\tprime \textbf{C}^{ij}[\textbf{X}] \textbf{V} \right)  \right\|_\textnormal{F}^2.
	\end{align*}
\end{lemma}

\begin{proof}[Proof of Lemma \ref{lemma:diag}]
	Note that for all $\textbf{V} \in \mathcal{U}$,
	\begin{align*}
	\left\|  \textbf{C}^{ij}[\textbf{X}]   \right\|_\textnormal{F}^2 = \left\| \textbf{V}\tprime \textbf{C}^{ij}[\textbf{X}]  \textbf{V} \right\|_\textnormal{F}^2 =  \left\|  \mathrm{diag} \left( \textbf{V}\tprime\textbf{C}^{ij}[\textbf{X}] \textbf{V} \right)\right\|_\textnormal{F}^2  +  \left\| \mathrm{off} \left( \textbf{V}\tprime \textbf{C}^{ij}[\textbf{X}]  \textbf{V}\right)  \right\|_\textnormal{F}^2.
	\end{align*}
	Thus, $\underset{\textbf{V}\in \mathcal{U}}{\arg\!\min} \left\{ \tilde{g}\left(\textbf{V},\textbf{X}\right) \right\}$  is equivalent to 
	\begin{align*}
	&\underset{\textbf{V}\in \mathcal{U}}{\arg\!\max}\left\{\sum_{i,j = 1}^p  \left( \left\|  \mathrm{diag} \left( \textbf{V}\tprime \textbf{C}^{ij}[\textbf{X}] \textbf{V} \right) \right\|_\textnormal{F}^2  - \left\| \textbf{C}^{ij}[\textbf{X}]   \right\|_\textnormal{F}^2\right)\right\} = \underset{\textbf{V}\in \mathcal{U}}{\arg\!\max}\left\{ {g}\left( \textbf{V},\textbf{X}\right) \right\}.
	\end{align*}
\end{proof}

\begin{lemma}\label{lem:H_simplification}
	For all $ (i, j) \in \{1, \ldots ,p\} \times \{1, \ldots ,p\}   $ and any orthogonal matrix $\textbf{H} = (h_{ij}) \in \mathcal{R}^{p \times p}$, the sample and the population cumulant matrices satisfy,
	\begin{align*}
	\textbf{H}\tprime \textbf{C}{}^{ij}[\textbf{H} \textbf{X}] \textbf{H} &= \sum_{a = 1}^p \sum_{b = 1}^p h_{ia} h_{jb} \textbf{C}{}^{ab}[ \textbf{X} ] \numberthis \label{eq:auxiliary}  \\
	\textbf{H}\tprime \hat{\textbf{C}}{}^{ij}[\textbf{H} \textbf{X}_N] \textbf{H} &= \sum_{a = 1}^p \sum_{b = 1}^p h_{ia} h_{jb} \hat{\textbf{C}}{}^{ab}[ \textbf{X}_N ],
	\end{align*}
	where $ \textbf{H} \textbf{X}_N  $ denotes the sample $ \{\textbf{H} \textbf{X}_1, \ldots , \textbf{H} \textbf{X}_n \}$.
\end{lemma}

\begin{proof}[Proof of Lemma \ref{lem:H_simplification}]
	Both results follow by the same arguments and we prove only the former. Since $ \textbf{e}_i\tprime \textbf{X} \textbf{X}\tprime \textbf{e}_j $ is a scalar and $\boldsymbol{\Sigma}_1 [\textbf{H} \textbf{X}] = \textbf{H} \boldsymbol{\Sigma}_1 [\textbf{X}] \textbf{H}\tprime $, the left-hand side of Eq. \eqref{eq:auxiliary} can be written as
	\begin{align}\label{eq:aux_1}
	\frac{1}{q} \mathbb{E} [ \textbf{e}_i\tprime \textbf{H} \textbf{X} \textbf{X}\tprime \textbf{H}\tprime \textbf{e}_j \textbf{X} \textbf{X}\tprime ] - \boldsymbol{\Sigma}_1 [\textbf{X}] \textbf{H}\tprime (\delta_{ij} q \textbf{I}_p + \textbf{E}^{ij} + \textbf{E}^{ji} ) \textbf{H} (\boldsymbol{\Sigma}_1 [\textbf{X}])\tprime,
	\end{align}
	where the first term can be written as
	\begin{align}\label{eq:aux_2}
	\sum_{a,b} \frac{1}{q} \mathbb{E} [ h_{ia} (\textbf{X} \textbf{X}\tprime)_{ab} h_{jb} \textbf{X} \textbf{X}\tprime ] = \sum_{a,b} h_{ia} h_{jb} \frac{1}{q} \mathbb{E} [ \textbf{e}_a\tprime \textbf{X} \textbf{X}\tprime \textbf{e}_b \textbf{X} \textbf{X}\tprime ], 
	\end{align}
	where the $ 1/q $-scaled expected value is the first term in the definition of the matrix $\textbf{C}^{ab}[\textbf{X}]$ and has the representation
	\begin{align}\label{eq:aux_3}
	\frac{1}{q} \mathbb{E} [ \textbf{e}_a\tprime \textbf{X} \textbf{X}\tprime \textbf{e}_b \textbf{X} \textbf{X}\tprime ] = \textbf{C}^{ab}[\textbf{X}] + \boldsymbol{\Sigma}_1 [\textbf{X}] (\delta_{ab} q \textbf{I}_p + \textbf{E}^{ab} + \textbf{E}^{ba} ) (\boldsymbol{\Sigma}_1 [\textbf{X}])\tprime.
	\end{align}
	Plugging Eq. \eqref{eq:aux_3} into Eq. \eqref{eq:aux_2} we obtain
	\begin{align*}
	&\sum_{a,b} \frac{1}{q} \mathbb{E} [ h_{ia} (\textbf{X} \textbf{X}\tprime)_{ab} h_{jb} \textbf{X} \textbf{X}\tprime ] \\
	=& \sum_{a,b} h_{ia} h_{jb} \left( \textbf{C}^{ab}[\textbf{X}] + \boldsymbol{\Sigma}_1 [\textbf{X}] (\delta_{ab} q \textbf{I}_p + \textbf{E}^{ab} + \textbf{E}^{ba} ) (\boldsymbol{\Sigma}_1 [\textbf{X}])\tprime \right). \numberthis \label{eq:aux_4}
	\end{align*}
	Furthermore, plugging Eq. \eqref{eq:aux_4} into Eq. \eqref{eq:aux_1} gives us,
	\begin{align}\label{eq:aux_5}
	\begin{split}
	\textbf{H}\tprime \textbf{C}{}^{ij}[\textbf{H} \textbf{X}] \textbf{H} &= \sum_{a,b} h_{ia} h_{jb} \left( \textbf{C}^{ab}[\textbf{X}] + \boldsymbol{\Sigma}_1 [\textbf{X}] (\delta_{ab} q \textbf{I}_p + \textbf{E}^{ab} + \textbf{E}^{ba} ) (\boldsymbol{\Sigma}_1 [\textbf{X}])\tprime \right) \\
	&- \boldsymbol{\Sigma}_1 [\textbf{X}] \textbf{H}\tprime (\delta_{ij} q \textbf{I}_p + \textbf{E}^{ij} + \textbf{E}^{ji} ) \textbf{H} (\boldsymbol{\Sigma}_1 [\textbf{X}])\tprime.
	\end{split}
	\end{align}
	
	Element-wise examination reveals that,
	\begin{align*}
	&\sum_{a,b} h_{ia} h_{jb} \boldsymbol{\Sigma}_1 [\textbf{X}] ( \delta_{ab} q \textbf{I}_p )  (\boldsymbol{\Sigma}_1 [\textbf{X}])\tprime - \boldsymbol{\Sigma}_1 [\textbf{X}] \textbf{H}\tprime (\delta_{ij} q \textbf{I}_p  ) \textbf{H} (\boldsymbol{\Sigma}_1 [\textbf{X}])\tprime = \textbf{0}, \\
	&\sum_{a,b} h_{ia} h_{jb} \boldsymbol{\Sigma}_1 [\textbf{X}]  \textbf{E}^{ab}   (\boldsymbol{\Sigma}_1 [\textbf{X}])\tprime - \boldsymbol{\Sigma}_1 [\textbf{X}] \textbf{H}\tprime \textbf{E}^{ij}  \textbf{H} (\boldsymbol{\Sigma}_1 [\textbf{X}])\tprime = \textbf{0}, \quad \textnormal{and} \\
	&\sum_{a,b} h_{ia} h_{jb} \boldsymbol{\Sigma}_1 [\textbf{X}]  \textbf{E}^{ba}   (\boldsymbol{\Sigma}_1 [\textbf{X}])\tprime - \boldsymbol{\Sigma}_1 [\textbf{X}] \textbf{H}\tprime \textbf{E}^{ji}  \textbf{H} (\boldsymbol{\Sigma}_1 [\textbf{X}])\tprime = \textbf{0}.
	\end{align*}
	This concludes the proof.
\end{proof}

\begin{lemma}\label{lem:U_convergence}
	Let $ \hat{\textbf{S}} \in \mathcal{R}^{p \times p}$ be a sequence of estimators (indexed by $ n $) such that $ \sqrt{n} ( \hat{\textbf{S}} - \boldsymbol{\Lambda}) \rightsquigarrow \mathcal{D}$, where $  \rightsquigarrow $ denotes convergence in distribution, $ \mathcal{D} $ is some matrix-valued distribution and
	\[ 
	\boldsymbol{\Lambda} = \begin{pmatrix}
	\lambda_1 \textbf{I}_{k_1} & \textbf{0} & \cdots & \textbf{0} \\
	\textbf{0} & \lambda_2 \textbf{I}_{k_2} & \cdots & \textbf{0} \\
	\vdots & \vdots & \ddots & \vdots \\
	\textbf{0} & \textbf{0} & \cdots & \lambda_R \textbf{I}_{k_R}
	\end{pmatrix},
	\]
	where  $ k_1 + \cdots + k_R = p $ and the values $ \lambda_k $, $ k \in \{ 1, \ldots , R \} $, are distinct and in a strictly decreasing order ($ \lambda_1 > \lambda_2 > \cdots > \lambda_R$). Let $ \hat{\textbf{U}} $ be the sequence of eigenvector matrices, where the columns are the eigenvectors of the matrices $\hat{\textbf{S}}$. Partition $\hat{\textbf{U}}$ into blocks $ \hat{\textbf{U}}_{ij} \in \mathcal{R}^{k_i \times k_j}$ as
	\begin{align*} 
	\hat{\textbf{U}} = \begin{pmatrix}
	\hat{\textbf{U}}_{11} & \hat{\textbf{U}}_{12} & \cdots & \hat{\textbf{U}}_{1R} \\
	\hat{\textbf{U}}_{21} & \hat{\textbf{U}}_{22} & \cdots & \hat{\textbf{U}}_{2R} \\
	\vdots & \vdots & \ddots & \vdots \\
	\hat{\textbf{U}}_{R1} & \hat{\textbf{U}}_{R2} & \cdots & \hat{\textbf{U}}_{RR}
	\end{pmatrix},
	\end{align*}
	in a similar way to $ \boldsymbol{\Lambda} $. Then
	\[ 
	\hat{\textbf{U}}_{ij} = \mathcal{O}_p \left( \frac{1}{\sqrt{n}} \right), \quad \mbox{whenever} \quad  i \neq j.
	\] 
\end{lemma}

\begin{proof}[Proof of Lemma \ref{lem:U_convergence}]
	
	The sequence of eigenvectors satisfies the eigenequation
	\begin{align}\label{eq:eigen_equation}
	\hat{\textbf{S}} \hat{\textbf{U}} = \hat{\textbf{U}} \hat{\boldsymbol{\Lambda}},
	\end{align}
	where $ \hat{\boldsymbol{\Lambda}} $ is a sequence of diagonal matrices containing the estimated eigenvalues. Then,
	\begin{align}\label{eq:eigen_equation2}
	\sqrt{n} \hat{\textbf{S}} \hat{\textbf{U}} = \sqrt{n} ( \hat{\textbf{S}} - \boldsymbol{\Lambda}) \hat{\textbf{U}} + \sqrt{n} \boldsymbol{\Lambda} \hat{\textbf{U}} = \sqrt{n} \boldsymbol{\Lambda} \hat{\textbf{U}} + \mathcal{O}_p(1),
	\end{align}
	where the second equality holds since $ \hat{\textbf{U}} = \mathcal{O}_p(1)$ as a result of the compactness of the space of orthogonal matrices and since\red{, by Prohorov's theorem, our assumption $ \sqrt{n} ( \hat{\textbf{S}} - \boldsymbol{\Lambda}) \rightsquigarrow \mathcal{D}$ implies that $ \sqrt{n} ( \hat{\textbf{S}} - \boldsymbol{\Lambda}) = \mathcal{O}_p(1) $.} 
	
	It follows from \cite[Theorem 3.2]{eaton1991wielandt} that $ \sqrt{n} (\hat{\boldsymbol{\Lambda}} - \boldsymbol{\Lambda}) = \mathcal{O}_p(1) $, and consequently we obtain
	\begin{align}\label{eq:eigen_equation3}
	\sqrt{n} \hat{\textbf{U}} \hat{\boldsymbol{\Lambda}} = \hat{\textbf{U}} \sqrt{n} (\hat{\boldsymbol{\Lambda}} - \boldsymbol{\Lambda}) + \sqrt{n} \hat{\textbf{U}} \boldsymbol{\Lambda} = \sqrt{n} \hat{\textbf{U}} \boldsymbol{\Lambda} + \mathcal{O}_p(1),
	\end{align}
	By combining Eqs. \eqref{eq:eigen_equation2} and \eqref{eq:eigen_equation3} with Eq. \eqref{eq:eigen_equation}, we arrive at
	\[ 
	\boldsymbol{\Lambda} \hat{\textbf{U}} - \hat{\textbf{U}} \boldsymbol{\Lambda} = \mathcal{O}_p \left( \frac{1}{\sqrt{n}} \right).
	\]
	If this is written block-wise, we obtain
	\[
	(\lambda_i - \lambda_j) \hat{\textbf{U}}_{ij} = \mathcal{O}_p \left( \frac{1}{\sqrt{n}} \right), \quad \mbox{for all} \quad   (i, j) \in \{ 1, \ldots , R \} \times \{ 1, \ldots , R \}  .
	\]
	The result now follows by considering only the off-diagonal blocks, $ i \neq j $, and dividing by the non-zero $ (\lambda_i - \lambda_j) $.
\end{proof}

\begin{remark}
	\red{In the sequel, the assumption of Lemma \ref{lem:U_convergence} that $ \sqrt{n} ( \hat{\textbf{S}} - \boldsymbol{\Lambda}) \rightsquigarrow \mathcal{D}$ is fulfilled by applying central limit theorem.}
\end{remark}

\begin{proof}[Proof of Theorem \ref{theo:functional}]
	Let Assumptions \ref{assu:JADE} and \ref{assu:FOBI}$ (v) $ hold for some fixed $ v $. Consider now the $ k $-TJADE functional $ \hat{\boldsymbol{\Gamma}}{}^k $ where $ k \geq v$.
	
	By \cite[Theorem 1]{virta2017independent} the standardized matrix satisfies $\textbf{X}^{\textnormal{st}} = \tau \textbf{U}_1 \textbf{Z} \textbf{U}_2\tprime$ for some $\tau > 0$ and for some orthogonal matrices, $\textbf{U}_1 \in \mathcal{R}^{p \times p}, \textbf{U}_2 \in \mathcal{R}^{q \times q}$. By \cite[Eq. (5)]{virta2017independent} the TFOBI-matrix functional $\textbf{B}[\textbf{X}^\textnormal{st}]$ has the form, 
	\[
	\textbf{B}[\textbf{X}^\textnormal{st}]  = \textbf{U}_1 \left( \sum_{k=1}^p \tau^4  \left( \kappa_k + p + q +1 \right) \textbf{E}^{kk} \right) \textbf{U}_1\tprime,
	\]
	where, $
	\kappa_k = \frac{1}{q} \sum_{j=1}^q \mathbb{E}[z_{kj}^4] -3,$
	are the row means of the kurtosis values of $ \textbf{Z} = (z_{ij}) $ such that $\kappa_1 \geq \cdots \geq \kappa_p$. Let $ R $ be the number of the distinct eigenvalues of $\textbf{B}[\textbf{X}^\textnormal{st}]$. Denote these distinct values by  $\lambda_1 > \cdots > \lambda_R$ and denote the corresponding multiplicities of these values by $k_1, \ldots , k_R$, respectively. Note that each $ \lambda_r $ can be given as $ \tau^4  \left( \kappa_k + p + q +1 \right) $ for some $ k $. As $ k \geq v$, where $ v $ is the largest multiplicity among the row mean kurtoses $ \kappa_1, \ldots ,\kappa_p $, we have that $ k_r \leq k $, for all $ r $. 
	
	The set of eigenvectors of $\textbf{B}[\textbf{X}^\textnormal{st}]$ is identifiable up to orthogonal transformations within each eigenspace. That is, ignoring the order and signs, the eigenvector matrix $\textbf{W}_1$ of $\textbf{B}[\textbf{X}^\textnormal{st}]$ has the form
	\begin{align}\label{eq:H_form}
	\textbf{W}_1 = \textbf{U}_1 \begin{pmatrix}
	\textbf{H}_{11}\tprime & \textbf{0} & \cdots & \textbf{0} \\
	\textbf{0} & \textbf{H}_{12}\tprime & \cdots & \textbf{0} \\
	\vdots & \vdots & \ddots & \vdots \\ 
	\textbf{0} & \textbf{0} & \cdots & \textbf{H}_{1R}\tprime
	\end{pmatrix} = \textbf{U}_1 \textbf{H}_1\tprime,
	\end{align}
	where each $\textbf{H}_{1r} \in \mathcal{R}^{k_r \times k_r}$, $r \in \{ 1, \ldots , R \}$, is orthogonal. The FOBI-rotated data is then
	\[
	\textbf{X}^{\textnormal{F}} = \boldsymbol{\Gamma}^\textnormal{F}\left[\textbf{X}\right] \textbf{X} (\boldsymbol{\Gamma}^\textnormal{F}[\textbf{X}\tprime])\tprime = \textbf{W}_1\tprime \textbf{X}^\textnormal{st} \textbf{W}_2 = \tau \textbf{H}_1 \textbf{Z} \textbf{H}_2\tprime.
	\]
	
	The proof of Theorem \ref{theo:functional} is divided into two parts. First, we prove that the condition \textit{(i)} in Definition \ref{def:ic_functional} holds and after that we prove that the condition \textit{(ii)} holds.
	
	\medskip
	
	\noindent\textbf{Condition} \textit{(i)}:
	
	\medskip
	
	The condition \textit{(i)} claims that $k$-TJADE can, under Assumptions \ref{assu:JADE} and \ref{assu:FOBI}$ (v) $, estimate the block diagonal orthogonal matrix $\textbf{H}_1$ up to the signs and the order of its columns. For convenience, we drop all subscripts referring to the side (left or right) of the model, e.g. in the following $\textbf{H}_{1r}$ is $\textbf{H}_r$ and $\textbf{H}_1$ is  $\textbf{H}$.
	
	Adapting the proof of \cite[Theorem 1]{virta2017jade}, we have that
	\begin{equation}\label{eq:cij_elements}
	\textbf{C}^{ij} [ \textbf{X}^{\textnormal{F}} ] = \textbf{H} \left( \sum_{k=1}^p \tau^4 h_{ik} h_{jk} \kappa_k \textbf{E}^{kk} \right) \textbf{H}\tprime = \textbf{H} \textbf{D}^{ij} \textbf{H}\tprime,
	\end{equation}
	where $\textbf{D}^{ij} =  \sum_{k=1}^p \tau^4 h_{ik} h_{jk} \kappa_k \textbf{E}^{kk} $ is diagonal matrix for every $i,j\in \{1,\ldots,p\}$. Since $\tilde{g}( \textbf{H}, \textbf{X}^\textnormal{F} ) = 0$, the matrix  $\textbf{H}$ is a solution to the k-TJADE  minimization problem
	\begin{align*}
	\underset{\{\textbf{V}:\textbf{V}\textbf{V}\tprime = \textbf{I}_p\}}{\mbox{min}}\left\{\tilde{g}( \textbf{V},\textbf{X}^\textnormal{F} ) \right\} = \underset{\{\textbf{V}:\textbf{V}\textbf{V}\tprime = \textbf{I}_p\}}{\mbox{min}}\left\{ \sum_{\left|i-j\right|<k}\left\| \mathrm{off} \left( \textbf{V}\tprime \textbf{C}^{ij}[\textbf{X}^\textnormal{F}] \textbf{V} \right)  \right\|_\textnormal{F}^2\right\}.
	\end{align*}
	
	Next, we show that $\textbf{H}$ is a unique minimizer, up to the signs and the order of its columns. Let $K_r = \sum_{m=1}^r k_r$ such that $K_0 = 0$ and let $\mathcal{I}_r = \{ K_{r-1} + 1, K_{r-1} + 2, \ldots , K_r \}$ be the subset of the index set $\{ 1, \ldots ,p \}$ for which the corresponding columns of $\textbf{H}$ belong to the $r$th eigenvalue block. 
	By \cite[Lemma 2]{bonhomme2009consistent}, the minimizer is unique up to the signs and order of its columns if, for each pair of distinct columns $\textbf{h}_s, \textbf{h}_t$ of $ \textbf{H} $, there exists a pair $ (i, j) \in \{ (i, j) : | i - j | < k \} $ such that the eigenvalues of $\textbf{C}^{ij}[\textbf{X}^\textnormal{F}]$ corresponding to $\textbf{h}_s$ and $\textbf{h}_t$ are distinct. By the decomposition in Eq. \eqref{eq:cij_elements}, this is equivalent to requiring that
	\begin{align}\label{eq:exists_pair}
	\exists (i, j): | i - j | < k, \quad \mbox{such that} \quad \tau^4 h_{is} h_{js} \kappa_s \neq \tau^4 h_{it} h_{jt} \kappa_t.
	\end{align} 
	We next show that Eq. \eqref{eq:exists_pair} holds for all $ s \neq t $ by considering separately the two cases where $ s $ and $ t $ either belong to two different sets or where $ s $ and $ t $ belong to the same set of the partition $\mathcal{I}_1, \ldots , \mathcal{I}_R$.
	
	First, assume that $ s $ and $ t $ belong to different sets of the partition. We proceed with proof by contraposition and assume that Eq. \eqref{eq:exists_pair} does not hold. That is, for all $ (i, j)$, $| i - j | < k $, the eigenvalue pairs are always equal. In particular, $ \tau^4 h_{is}^2 \kappa_s = \tau^4 h_{it}^2 \kappa_t $ for all $ i \in \{ 1, \ldots , p \} $. By summing over the $ p $ equations, we have, from the orthogonality of $ \textbf{H} $, that $ \tau^4 \kappa_s = \tau^4 \kappa_t$, where $\tau^4 >0$. This implies  that $ \kappa_s = \kappa_t $, which is a contradiction since we assumed that $ s $ and $ t $ belong to different sets of the partition and thus have different eigenvalues. Consequently, Eq. \eqref{eq:exists_pair} holds for any pair of columns belonging to distinct sets of the index partition. 
	
	Assume then that $ s $ and $ t $ belong to the same set $ \mathcal{I}_q $ of the partition. We again proceed with proof by contraposition and assume that Eq. \eqref{eq:exists_pair} does not hold. Now, $ \kappa_s = \kappa_t $ and we have that $\tau^4 h_{is} h_{js} \kappa_s = \tau^4 h_{it} h_{jt} \kappa_s$ for all $ (i, j)$, $| i - j | < k $. In particular, this holds for all $ (i, j) $ in the subset
	\begin{align}\label{eq:set_intersection} 
	\left\{(i, j) :| i - j | < k \right\} \cap \left\{ (i, j) : i \in  \mathcal{I}_q \land j \in  \mathcal{I}_q  \right\}.
	\end{align}
	By Assumption \ref{assu:FOBI}$ (v) $ $k_q \leq k$ and consequently the distance $ | i - j | $  is always less than $ k  $ in the set $  \left\{ (i, j) : i \in  \mathcal{I}_q \land j \in  \mathcal{I}_q  \right\} $. Hereby, the intersection in Eq. \eqref{eq:set_intersection} is equal to $  \left\{ (i, j) : i \in  \mathcal{I}_q \land j \in  \mathcal{I}_q  \right\} $. We now multiply each equality $\tau^4 h_{is} h_{js} \kappa_s = \tau^4 h_{it} h_{jt} \kappa_s$, with indices $ (i, j) $ in the set $ \left\{ (i, j) : i \in  \mathcal{I}_q \land j \in  \mathcal{I}_q  \right\} $, by $ h_{is} h_{it} $. By summing twice over $ \mathcal{I}_q $, we obtain
	\begin{align}\label{eq:eigen_equality} 
	\tau^4 \kappa_s \left( \sum_{i \in \mathcal{I}_q} h_{is}^2 \right) \left( \sum_{j \in \mathcal{I}_q} h_{js}^2 \right) = \tau^4 \kappa_s \left( \sum_{i \in \mathcal{I}_q} h_{is} h_{it} \right) \left( \sum_{j \in \mathcal{I}_q} h_{js} h_{jt} \right).
	\end{align}
	The constant $ \tau^4 > 0$. The constant $ \kappa_s \neq 0$, as the contrary would imply that two of the kurtosis values were equal to zero, $ \kappa_s = \kappa_t = 0 $, which would then contradict Assumption~\ref{assu:JADE}. Thus we can divide both sides of Eq. \eqref{eq:eigen_equality} by $ \tau^4 \kappa_s $ and we obtain
	\begin{align}\label{eq:eigen_equality_2} 
	\left( \sum_{i \in \mathcal{I}_q} h_{is}^2 \right) \left( \sum_{j \in \mathcal{I}_q} h_{js}^2 \right) = \left( \sum_{i \in \mathcal{I}_q} h_{is} h_{it} \right) \left( \sum_{j \in \mathcal{I}_q} h_{js} h_{jt} \right).
	\end{align}
	As $ \textbf{H} $ has the block diagonal structure given in Eq. \eqref{eq:H_form}, the sums in Eq. \eqref{eq:eigen_equality_2} are dot products between the columns of the $ q $th orthogonal block of $ \textbf{H} $ and the left-hand side of Eq. \eqref{eq:eigen_equality_2} is equal to one and the right-hand side of Eq. \eqref{eq:eigen_equality_2} is equal to zero. This is obviously a contradiction. Consequently, Eq. \eqref{eq:exists_pair} holds also for any pair of columns which belong to the same set of the partition This concludes the proof of condition~\textit{(i)}.
	
	\medskip
	
	\noindent\textbf{Condition} \textit{(ii)}:
	
	\medskip
	
	To see that condition \textit{(ii)} holds, recall from \cite{virta2017independent} that the TFOBI functional $\boldsymbol{\Gamma}^\textnormal{F}$ is orthogonally equivariant and that the TFOBI-transformation is orthogonally invariant.  Thus, for all $\textbf{X} \in \mathcal{R}^{p \times q}$ and all orthogonal $\textbf{U}_1 \in \mathcal{R}^{p \times p}$, $\textbf{U}_2 \in \mathcal{R}^{q \times q}$, we have that
	\[
	\boldsymbol{\Gamma}^k [\textbf{U}_1 \textbf{X} \textbf{U}_2\tprime ] = \textbf{V}[(\textbf{U}_1 \textbf{X} \textbf{U}_2)^\textnormal{F}] \boldsymbol{\Gamma}^\textnormal{F}[\textbf{U}_1 \textbf{X} \textbf{U}_2\tprime] \equiv \textbf{V}[\textbf{X}^\textnormal{F}] \boldsymbol{\Gamma}^\textnormal{F}[\textbf{X}] \textbf{U}_1\tprime = \boldsymbol{\Gamma}^k\left[\textbf{X}\right] \textbf{U}_1\tprime.
	\]
	This concludes the proof of condition \textit{(ii)}.
\end{proof}


\begin{proof}[Proof of Theorem \ref{theo:consistency}]

	Recall that $\textbf{X}_N = \{ \textbf{X}_1\ldots \textbf{X}_n \}$ is an i.i.d. sequence from the tensor IC model and let $\bar{\textbf{X}}$  be the sample mean of $\textbf{X}_N$. Then, the TFOBI-transformed observations are 
	\[
	\textbf{X}_i^\textnormal{F} = \hat{\textbf{H}} ( \hat{\boldsymbol{\Sigma}}_1[\textbf{X}_N])^{-\frac{1}{2}} \left( \textbf{X}_i - \bar{\textbf{X}} \right)  ( \hat{\boldsymbol{\Sigma}}_2[\textbf{X}_N])^{-\frac{1}{2}} \hat{\textbf{R}}{}\tprime = \hat{\textbf{H}} \textbf{Y}_i \hat{\textbf{R}}{}\tprime, \quad i \in \{1,\ldots,n\},
	\]
	where the orthogonal matrices $\hat{\textbf{H}}, \hat{\textbf{R}}$ are the left and right TFOBI-rotations and where $(\hat{\boldsymbol{\Sigma}}_{j}[\textbf{X}_N])^{-\frac{1}{2}}$, $j \in \{1,2\}$,  are the symmetric square roots of the left and right sample covariance matrices. Note that the cumulant matrices $ \hat{\textbf{C}}{}^{ij} $ depend on the observations $\textbf{X}_i^\textnormal{F}$ only through the product
	\[
	\textbf{X}_i^\textnormal{F} (\textbf{X}_i^\textnormal{F})\tprime = \hat{\textbf{H}} \textbf{Y}_i \hat{\textbf{R}}{}\tprime \hat{\textbf{R}} \textbf{Y}_i\tprime \hat{\textbf{H}}{}\tprime = \hat{\textbf{H}} \textbf{Y}_i \textbf{Y}_i\tprime \hat{\textbf{H}}{}\tprime,
	\]
	allowing us to omit the right rotation $\hat{\textbf{R}}$  in the following. 
	
	What makes proving the limiting results challenging, is that the sequence of matrices $\hat{\textbf{H}}$ has no general limiting properties. If there are any kurtosis values having multiplicity larger than one, this implies that the corresponding eigenvectors are not uniquely defined and, consequently, the TFOBI-solution for them does not converge. However, we can still say two things: First, by the compactness of the set of orthogonal matrices, $\hat{\textbf{H}} = \mathcal{O}_p(1)$. Second, by Lemma \ref{lem:U_convergence} and the central limit theorem, the elements of any column of $\hat{\textbf{H}}$ with indices not belonging to the corresponding diagonal multiplicity block converge to $0$ with the rate of root-$n$. That is, $\sqrt{n} \hat{h}_{kl} = \mathcal{O}_p(1)$, for $k, l$ satisfying the aforementioned conditions. These two, in conjunction with Assumption \ref{assu:FOBI}, are sufficient for proving the limiting results.
	
	As our first task, we show that the sample and the population objective functions of $ k $-TJADE are asymptotically equivalent to those of TJADE. The left $k$-TJADE estimator is $\hat{\textbf{V}}{}\tprime \hat{\textbf{H}}( \hat{\boldsymbol{\Sigma}}_1\left[\textbf{X}_N\right])^{-\frac{1}{2}}$, where the orthogonal $\hat{\textbf{V}} = (\hat{\textbf{v}}_1, \ldots , \hat{\textbf{v}}_p)$ is, by Lemma \ref{lemma:diag}, the sequence of the maximizers of the sequence of objective functions,
	\[
	\hat{g}(\textbf{V}, \hat{\textbf{H}} \textbf{Y}_N)=\sum_{| i - j | < k}^p \left\| \mathrm{diag} \left( \textbf{V}\tprime \hat{\textbf{C}}^{ij}[\hat{\textbf{H}} \textbf{Y}_N] \textbf{V} \right)  \right\|_\textnormal{F}^2   =  \sum_{| i - j | < k}^p \sum_{l=1}^p \left( \textbf{v}_l\tprime \hat{\textbf{C}}^{ij} [ \hat{\textbf{H}} \textbf{Y}_N ] \textbf{v}_l \right)^2.
	\]
	Consistency of the estimator $\hat{\textbf{V}}{}\tprime \hat{\textbf{H}}( \hat{\boldsymbol{\Sigma}}_1\left[\textbf{X}_N\right])^{-\frac{1}{2}}$ is now equal to the claim that there exists a sequence of estimators $\hat{\textbf{V}}{}\tprime \hat{\textbf{H}}( \hat{\boldsymbol{\Sigma}}_1\left[\textbf{X}_N\right])^{-\frac{1}{2}}$ that converges in probability to $\textbf{I}_p$. It is sufficient to show that $\hat{\textbf{V}}{}\tprime \hat{\textbf{H}} \rightarrow_{\mathbb{P}}  \textbf{I}_p$, since the weak law of large numbers and the continuous mapping theorem directly imply that $(\hat{\boldsymbol{\Sigma}}_1\left[\textbf{X}_N\right])^{-\frac{1}{2}} \rightarrow_{\mathbb{P}}  \textbf{I}_p$. 
	
	
	Since $ \hat{\textbf{V}} $ maximizes the objective function $ \hat{g} $ and since $\hat{\textbf{H}}$ is orthogonal, we have that
	\begin{align*}
	\hat{g}(\hat{\textbf{V}}, \hat{\textbf{H}} \textbf{Y}_N) &\geq \hat{g}(\textbf{V}, \hat{\textbf{H}} \textbf{Y}_N), \quad \forall \ \textbf{V} \mbox{ orthogonal}, \quad \mbox{and} \\
	\hat{g}(\hat{\textbf{H}} \hat{\textbf{H}}{}\tprime \hat{\textbf{V}}, \hat{\textbf{H}} \textbf{Y}_N) &\geq \hat{g}(\hat{\textbf{H}} \hat{\textbf{H}}{}\tprime \textbf{V}, \hat{\textbf{H}} \textbf{Y}_N), \quad \forall \ \textbf{V} \mbox{ orthogonal}.
	\end{align*}
	The range of $\textbf{V} \mapsto \hat{\textbf{H}}{}\tprime \textbf{V}$ is the set of all $p \times p$ orthogonal matrices and consequently $\hat{\textbf{H}}{}\tprime \hat{\textbf{V}} = \hat{\textbf{W}}$ is the sequence of maximizers for the argument-modified sequence of objective functions, $\hat{g}(\hat{\textbf{H}} {\textbf{W}}, \hat{\textbf{H}} \textbf{Y}_N)$. The original sequence of maximizers can be written as $\hat{\textbf{V}} =  \hat{\textbf{H}} \hat{\textbf{W}}$ and the claim takes the  form $\hat{\textbf{W}} \rightarrow_{\mathbb{P}}  \textbf{I}_p$. Applying Lemma \ref{lem:H_simplification}, this new sequence of objective functions $\hat{g}(\hat{\textbf{H}} {\textbf{W}}, \hat{\textbf{H}} \textbf{Y}_N)$ can be reformulated as
	\begin{align*}
	\hat{g}(\hat{\textbf{H}} \textbf{W}, \hat{\textbf{H}} \textbf{Y}_N) &= \sum_{| i - j | < k}^p \sum_{l=1}^p \left( \textbf{w}_l\tprime \hat{\textbf{H}}{}\tprime \hat{\textbf{C}}^{ij} [ \hat{\textbf{H}} \textbf{Y}_N ] \hat{\textbf{H}} \textbf{w}_l \right)^2 \\
	&= \sum_{| i - j | < k}^p \sum_{l=1}^p \sum_{a,b} \sum_{a', b'}  \hat{h}_{ia} \hat{h}_{jb} \hat{h}_{ia'} \hat{h}_{jb'} \textbf{w}_l\tprime \hat{\textbf{C}}^{ab}[  \textbf{Y}_N ] \textbf{w}_l \textbf{w}_l\tprime \hat{\textbf{C}}^{a'b'}[  \textbf{Y}_N ] \textbf{w}_l.
	\end{align*}
	Next, we add and subtract,
	\begin{align}\label{eq:leftover}
	\hat{m} = \sum_{| i - j | \geq k}^p \sum_{l=1}^p \sum_{a,b} \sum_{a', b'}  \hat{h}_{ia} \hat{h}_{jb} \hat{h}_{ia'} \hat{h}_{jb'} \textbf{w}_l\tprime \hat{\textbf{C}}^{ab}[ \textbf{Y}_N ] \textbf{w}_l \textbf{w}_l\tprime \hat{\textbf{C}}^{a'b'}[ \textbf{Y}_N ] \textbf{w}_l
	\end{align}
	in order to make the first sum run over the whole range of $i$ and $j$. This gives us
	\begin{align*}
	&\hat{g}(\hat{\textbf{H}} \textbf{W}, \hat{\textbf{H}} \textbf{Y}_N)
	= \sum_{i = 1}^p \sum_{j = 1}^p \sum_{l=1}^p \sum_{a,b} \sum_{a', b'}  \hat{h}_{ia} \hat{h}_{jb} \hat{h}_{ia'} \hat{h}_{jb'} \textbf{w}_l\tprime \hat{\textbf{C}}^{ab}[  \textbf{Y}_N ] \textbf{w}_l \textbf{w}_l\tprime \hat{\textbf{C}}^{a'b'}[  \textbf{Y}_N ] \textbf{w}_l - \hat{m} \\
	&= \sum_{l=1}^p \sum_{a,b} \sum_{a', b'} \sum_{i = 1}^p \left( \hat{h}_{ia} \hat{h}_{ia'} \right) \sum_{j = 1}^p \left( \hat{h}_{jb}  \hat{h}_{jb'} \right) \textbf{w}_l\tprime \hat{\textbf{C}}^{ab}[  \textbf{Y}_N ] \textbf{w}_l \textbf{w}_l\tprime \hat{\textbf{C}}^{a'b'}[  \textbf{Y}_N ] \textbf{w}_l - \hat{m} \\
	&= \sum_{l=1}^p \sum_{a,b} \sum_{a', b'} \delta_{aa'} \delta_{bb'} \textbf{w}_l\tprime \hat{\textbf{C}}^{ab}[  \textbf{Y}_N ] \textbf{w}_l \textbf{w}_l\tprime \hat{\textbf{C}}^{a'b'}[  \textbf{Y}_N ] \textbf{w}_l - \hat{m} \\
	&= \sum_{l=1}^p \sum_{a,b} \textbf{w}_l\tprime \hat{\textbf{C}}^{ab}[  \textbf{Y}_N ] \textbf{w}_l \textbf{w}_l\tprime \hat{\textbf{C}}^{ab}[  \textbf{Y}_N ] \textbf{w}_l - \hat{m} \\
	&= \sum_{a, b} \sum_{l=1}^p \left( \textbf{w}_l\tprime \hat{\textbf{C}}^{ab} [ \textbf{Y}_N ]  \textbf{w}_l \right)^2 - \hat{m},
	\end{align*}
	where the third equality follows from the orthogonality of $ \hat{\textbf{H}} $.
	
	Next, we establish two properties of $ \hat{m} $: $(1^\circ)$ If the argument $\textbf{W}$ is replaced with any sequence of orthogonal matrices bounded in probability (the boundedness is an instant consequence of the orthogonality), $n \cdot  \hat{m} $ is bounded in probability.  $(2^\circ)$ $ \hat{m} $ converges uniformly in probability to zero.
	
	We first prove the property $(1^\circ)$ (which will be used in the proof of Theorem \ref{theo:limiting} later). Let
	\[ 
	\tilde{m} = \sum_{| i - j | \geq k}^p \sum_{l=1}^p \sum_{a,b} \sum_{a', b'}  \hat{h}_{ia} \hat{h}_{jb} \hat{h}_{ia'} \hat{h}_{jb'} \tilde{\textbf{w}}_l\tprime \hat{\textbf{C}}^{ab}[ \textbf{Y}_N ] \tilde{\textbf{w}}_l \tilde{\textbf{w}}_l\tprime \hat{\textbf{C}}^{a'b'}[ \textbf{Y}_N ] \tilde{\textbf{w}}_l,
	\]
	where $ \tilde{\textbf{W}} = (\tilde{\textbf{w}}_1, \ldots , \tilde{\textbf{w}}_p) $ is some sequence of orthogonal matrices bounded in probability and all other terms are as in the definition of $ \hat{m} $ in Eq. \eqref{eq:leftover}. We divide the terms of $\tilde{m}$ into different cases based on the indices.
	
	\begin{enumerate}
		\item $ a \neq b $, $ a' \neq b' $: By the supplementary material of \cite{virta2017jade} and by the central limit theorem that $\hat{\textbf{C}}{}^{ab}[ \textbf{Y}_N ] = \mathcal{O}_p(1/\sqrt{n})$ and that $\hat{\textbf{C}}{}^{a'b'}[ \textbf{Y}_N ] = \mathcal{O}_p(1/\sqrt{n})$. As the elements $\hat{h}_{ia}, \hat{h}_{jb}, \hat{h}_{ia'}, \hat{h}_{jb'}, \tilde{\textbf{w}}_l$ are bounded in probability for all indices $ i,j,a,b,a',b',l$, any summand in $ \tilde{m} $ with the indices $ a \neq b $ and $ a' \neq b' $ is $\mathcal{O}_p(1/n)$.
		\item $ a \neq b $, $ a' = b' $: By the same arguments as in case 1, we have that $\hat{\textbf{C}}{}^{ab}[ \textbf{Y}_N ] = \mathcal{O}_p(1/\sqrt{n})$. Furthermore, the supplementary material of \cite{virta2017jade} provides us the result $\hat{\textbf{C}}{}^{a'b'}[ \textbf{Y}_N ] \rightarrow_{\mathbb{P}} \textbf{C}{}^{a' b'} $ for some constant matrix $ \textbf{C}{}^{a' b'} $. Both $\hat{h}_{ia'}$ and $\hat{h}_{ja'}$ (with $ | i - j | \geq k $) cannot belong to a diagonal block of $\hat{\textbf{H}}$ as that would compromise Assumption \ref{assu:FOBI}$ (v) $. Thus, by Lemma \ref{lem:U_convergence}, at least one of $\hat{h}_{ia'}$ and $\hat{h}_{ja'}$ must be $\mathcal{O}_p(1/\sqrt{n})$. Consequently, any summand in $ \tilde{m} $ with the indices $ a \neq b $ and $ a' = b' $ is $\mathcal{O}_p(1/n)$.
		\item $ a = b $, $ a' \neq b' $: Since $ \tilde{m} $ is symmetric with respect to $ (a, b) $ and $ (a', b') $, it follows from the same arguments as in case 2 that any summand in $ \tilde{m} $ with the indices $ a = b $ and $ a' \neq b' $ is $\mathcal{O}_p(1/n)$.
		\item $ a = b $, $ a' = b' $: By the same arguments as above, $\hat{\textbf{C}}{}^{ab}[ \textbf{Y}_N ]$ and $\hat{\textbf{C}}{}^{a'b'}[ \textbf{Y}_N ]$ both converge in probability to some constant matrices and the pairs $(\hat{h}_{ia'}, \hat{h}_{ja'})$ and $(\hat{h}_{ia}, \hat{h}_{ja})$ each provide one term that is $ \mathcal{O}_p(1/\sqrt{n}) $. Thus any summand in $ \tilde{m} $ with the indices $ a = b $ and $ a' = b' $ is $\mathcal{O}_p(1/n)$.
	\end{enumerate}
	
	Since the sum defining $ \tilde{m} $ is a finite sum of $ \mathcal{O}_p(1/n) $-terms, it holds that $ \tilde{m} = \mathcal{O}_p(1/n) $. Note that by choosing $ \tilde{\textbf{W}} = \textbf{W} $, we have $ \tilde{m} = \hat{m} $ and the objective function satisfies
	\[
	\hat{g}(\hat{\textbf{H}} \textbf{W}, \hat{\textbf{H}} \textbf{Y}_N) = \sum_{a,b}^p \sum_{l=1}^p \left( \textbf{w}_l\tprime \hat{\textbf{C}}^{ab} [ \textbf{Y}_N ]  \textbf{w}_l \right)^2 + \mathcal{O}_p(1/n) = \hat{f}(\textbf{W}, \textbf{Y}_N) + \mathcal{O}_p(1/n),
	\]
	where $ \hat{f}(\textbf{W}, \textbf{Y}) = \sum_{a,b}^p \sum_{l=1}^p ( \textbf{w}_l\tprime \hat{\textbf{C}}^{ab} [ \textbf{Y} ]  \textbf{w}_l )^2 $.
	
	We next prove property $(2^\circ)$. By the triangle inequality and the monotonicity of supremum,
	\[ 
	\sup_{\textbf{W}\in \mathcal{U}}  \left\{ | \hat{m} | \right\} \leq \sum_{| i - j | \geq k}^p \sum_{l=1}^p \sum_{a,b} \sum_{a', b'} \sup_{\textbf{W}\in \mathcal{U}}  \left\{ \left| \hat{h}_{ia} \hat{h}_{jb} \hat{h}_{ia'} \hat{h}_{jb'} \textbf{w}_l\tprime \hat{\textbf{C}}^{ab}[ \textbf{Y}_N ] \textbf{w}_l \textbf{w}_l\tprime \hat{\textbf{C}}^{a'b'}[ \textbf{Y}_N ] \textbf{w}_l \right| \right\},
	\]
	where $\mathcal{U} = \{ \textbf{V} \in \mathcal{R}^{p\times p}   : \textbf{VV}\tprime = \textbf{I}_p  \} $. It is thus sufficient to show that 
	\[
	\sup_{\textbf{W}\in \mathcal{U}}  \left\{ \left| \hat{h}_{ia} \hat{h}_{jb} \hat{h}_{ia'} \hat{h}_{jb'} \textbf{w}_l\tprime \hat{\textbf{C}}^{ab}[ \textbf{Y}_N ] \textbf{w}_l \textbf{w}_l\tprime \hat{\textbf{C}}^{a'b'}[ \textbf{Y}_N ] \textbf{w}_l \right| \right\} = o_p(1)
	\]
	individually for all terms in the sum $\hat{m}$. By the Cauchy-Schwarz inequality, the operator norm inequality, the equivalence of all finite-dimensional norms and the monotonicity of the supremum, we obtain the following upper bound (where the dependency of the expression on the variable $\textbf{W}$ is removed by the unit lengths of its columns),
	\begin{align*}
	&\sup_{\textbf{W}\in \mathcal{U}}  \left\{\left| \hat{h}_{ia} \hat{h}_{jb} \hat{h}_{ia'} \hat{h}_{jb'} \textbf{w}_l\tprime \hat{\textbf{C}}^{ab}[ \textbf{Y}_N ] \textbf{w}_l \textbf{w}_l\tprime \hat{\textbf{C}}^{a'b'}[ \textbf{Y}_N ] \textbf{w}_l \right|  \right\} \\
	\leq &\sup_{\textbf{W}\in \mathcal{U}}  \left\{ |\hat{h}_{ia} \hat{h}_{jb} \hat{h}_{ia'} \hat{h}_{jb'}| \| \textbf{w}_l \|_\textnormal{F} \| \hat{\textbf{C}}^{ab}[ \textbf{Y}_N ] \textbf{w}_l \|_\textnormal{F} \| \textbf{w}_l \|_\textnormal{F} \| \hat{\textbf{C}}^{a'b'}[ \textbf{Y}_N ] \textbf{w}_l \|_\textnormal{F}  \right\} \\
	\leq &\sup_{\textbf{W}\in \mathcal{U}}  \left\{ |\hat{h}_{ia} \hat{h}_{jb} \hat{h}_{ia'} \hat{h}_{jb'}| \| \hat{\textbf{C}}^{ab}[ \textbf{Y}_N ] \|_\textnormal{2} \| \textbf{w}_l \|_\textnormal{F}   \| \hat{\textbf{C}}^{a'b'}[ \textbf{Y}_N ] \|_\textnormal{2} \| \textbf{w}_l \|_\textnormal{F}  \right\} \\
	\leq & c_0 |\hat{h}_{ia} \hat{h}_{jb} \hat{h}_{ia'} \hat{h}_{jb'}| \| \hat{\textbf{C}}^{ab}[ \textbf{Y}_N ] \|_\textnormal{F} \| \hat{\textbf{C}}^{a'b'}[ \textbf{Y}_N ] \|_\textnormal{F}, \label{eq:sup_bound} \numberthis
	\end{align*}
	where $ c_0 $ is a constant resulting from switching from the operator norm to the Frobenius norm. Using the same arguments we used to prove the property $(1^\circ)$ above, Eq. \eqref{eq:sup_bound} can be seen to always converge in probability to zero, regardless of the values of the indices. Thus $\hat{g}(\hat{\textbf{H}} \textbf{W}, \hat{\textbf{H}} \textbf{Y}_N) = \hat{f}(\textbf{W}, \textbf{Y}_N) - \hat{m}$ for which $\sup_\textbf{W} \{ | \hat{m} | \} = o_p(1)$.
	
	Now, Lemma \ref{lem:H_simplification} along with exactly analogous calculations as above (with exact zeros taking the roles of $o_p(1)$-terms) can be used to show that under Assumption \ref{assu:FOBI}$ (v) $, the population version of the argument-modified $k$-TJADE objective function,
	\[
	g(\textbf{H} \textbf{W}, \textbf{H} \textbf{Y}) =  \sum_{| i - j | < k}^p \sum_{l=1}^p \left( \textbf{w}_l\tprime \textbf{H}\tprime \textbf{C}^{ij} [ \textbf{H} \textbf{Y} ] \textbf{H} \textbf{w}_l \right)^2,
	\]
	where $\textbf{H}$ is the population TFOBI-rotation, equals the augmented objective function not depending on $\textbf{H}$, that is,
	\[
	f(\textbf{W}, \textbf{Y}) := \sum_{i, j}^p \sum_{l=1}^p \left( \textbf{w}_l\tprime \textbf{C}^{ij} [ \textbf{Y} ] \textbf{w}_l \right)^2 = g(\textbf{H} \textbf{W}, \textbf{H} \textbf{Y}).
	\]
	Consequently,
	\begin{align}\label{eq:jade_equivalence}
	\hat{g}(\hat{\textbf{H}} \textbf{W}, \hat{\textbf{H}} \textbf{Y}_N) - g(\textbf{H} \textbf{W}, \textbf{H} \textbf{Y}) = \hat{f}(\textbf{W}, \textbf{Y}_N) - f(\textbf{W}, \textbf{Y}) - \hat{m}, 
	\end{align}
	where $\hat{f}(\textbf{W}, \textbf{Y}_N)$ and $f(\textbf{W}, \textbf{Y})$ are the sample and population objective functions of TJADE \citep{virta2017jade}.
	
	Having established this, we next show that the difference in the population and the sample $ k $-TJADE objective functions converges uniformly in probability to zero, allowing us to use the M-estimator convergence argument, see e.g. \cite[Theorem 5.7]{van1998asymptotic}, to prove the consistency of the estimator. By Eq. \eqref{eq:jade_equivalence},
	we have
	\begin{align}\label{eq:g_to_f_convergence}
	\begin{split}
	&\sup_{\textbf{W}\in \mathcal{U}}\left\{\left| \hat{g}(\hat{\textbf{H}} \textbf{W}, \hat{\textbf{H}} \textbf{Y}_N) - g(\textbf{H} \textbf{W}, \textbf{H} \textbf{Y}) \right|\right\}\\
	=& 
	\sup_{\textbf{W}\in \mathcal{U}}\left\{ \left| \hat{f}(\textbf{W}, \textbf{Y}_N) - f(\textbf{W}, \textbf{Y}) - \hat{m} \right| \right\} \\
	\leq& 
	\sup_{\textbf{W}\in \mathcal{U}} \left\{  \left| \hat{f}(\textbf{W}, \textbf{Y}_N) - f(\textbf{W}, \textbf{Y}) \right| \right\}  + o_p(1),
	\end{split}
	\end{align}
	and the uniform convergence of the $ k $-TJADE objective functions is a direct consequence of the uniform convergence of the TJADE objective functions. The latter result is given implicitly in \cite{virta2017jade}, but for completeness, we present the proof here.
	
	Let,
	\begin{align*}
	\textbf{A}^{ij} = \textnormal{diag}\left(\textbf{W}\tprime \textbf{C}^{ij}[\textbf{Y}] \textbf{W}\right) 
	\quad \textnormal{ and } \quad
	\hat{\textbf{A}}^{ij} = \textnormal{diag}\left(\textbf{W}\tprime \hat{\textbf{C}}^{ij}[\textbf{Y}_N] \textbf{W}\right).
	\end{align*}
	Since the Frobenius norm of the diagonal elements of a matrix is always bounded by the Frobenius norm of the entire matrix and since $ \textbf{W} $ is orthogonal, the following inequalities hold,
	\begin{align*}
	&\|   \textbf{A}^{ij}\|_\textnormal{F}  \leq \|   \textbf{C}^{ij}[\textbf{Y}]\|_\textnormal{F} \textnormal{, } \quad  \|   \hat{\textbf{A}}^{ij}\|_\textnormal{F}  \leq \|   \hat{\textbf{C}}^{ij}[\textbf{Y}_N] \|_\textnormal{F} \quad \textnormal{ and }\\
	& \left\|   \textbf{A}^{ij}- \hat{\textbf{A}}^{ij}\right\|_\textnormal{F} = 
	\left\| \textnormal{diag}\left( \textbf{W}\tprime( \textbf{C}^{ij}[\textbf{Y}] - \hat{\textbf{C}}^{ij}[\textbf{Y}_N])\textbf{W}  \right) \right\|_\textnormal{F}
	\leq   \left\|   \textbf{C}^{ij}[\textbf{Y}] - \hat{\textbf{C}}^{ij}[\textbf{Y}_N] \right\|_\textnormal{F}.
	\end{align*}
	By the reverse triangle inequality, we have that for any $ \textbf{A}, \textbf{B} \in \mathcal{R}^{p \times p} $,
	\begin{align*}
	&\left| \|   \textbf{A} \|_\textnormal{F}^2    - \| \textbf{B} \|_\textnormal{F}^2   \right| = \left| ( \|   \textbf{A} \|_\textnormal{F}    - \| \textbf{B} \|_\textnormal{F}) ( \|\textbf{A} \|_\textnormal{F}    + \| \textbf{B} \|_\textnormal{F})  \right| \leq \| \textbf{A} -  \textbf{B} \|_\textnormal{F} \left(  \| \textbf{A} \|_\textnormal{F} +  \| \textbf{B} \|_\textnormal{F} \right).
	\end{align*}
	Now, by the monotonicity of the supremum, 
	\begin{align}
	\label{eq:supconv}
	\begin{split}
	&\underset{\textbf{W} \in \mathcal{U} } {\sup}\left\{| f (\textbf{W},\textbf{Y})  - \hat{f} (\textbf{W},\textbf{Y}_N ) |\right\}
	\leq  
	\underset{\textbf{W} \in \mathcal{U} } {\sup} \left\{ \sum_{i,j}^p \left| \| \textbf{A}^{ij} \|_\textnormal{F}^2 - \| \hat{\textbf{A}}^{ij} \|_\textnormal{F}^2   \right| \right\} \\
	& \leq 
	\underset{\textbf{W} \in \mathcal{U} } {\sup} \left\{ \sum_{i,j}^p  \left\| \textbf{A}^{ij} - \hat{\textbf{A}}^{ij} \right\|_\textnormal{F}\left(\| \textbf{A}^{ij}\|_\textnormal{F} + \|\hat{\textbf{A}}^{ij} \|_\textnormal{F} \right) \right\} \\
	& \leq 
	\sum_{i,j}^p  \left\| \textbf{C}^{ij}[\textbf{Y}] - \hat{\textbf{C}}^{ij}[\textbf{Y}_N] \right\|_\textnormal{F}\left(\| \textbf{C}^{ij}[\textbf{Y}]\|_\textnormal{F} + \|\hat{\textbf{C}}^{ij}[\textbf{Y}_N] \|_\textnormal{F}\right),
	\end{split}
	\end{align}
	which converges in probability to zero, since $\hat{\textbf{C}}^{ij}[\textbf{Y}_N] \rightarrow_{\mathbb{P}} \textbf{C}^{ij}[\textbf{Y}]$.  Thus, the sequence of the sample TJADE objective functions $\hat{f}(\textbf{W},\textbf{Y}_N)$ converges uniformly in probability, with respect to the set of orthogonal matrices, to the theoretical TJADE objective function $f(\textbf{W},\textbf{Y})$. It now follows from Eqs. \eqref{eq:g_to_f_convergence} and \eqref{eq:supconv}, that under Assumption \ref{assu:JADE} and Assumption \ref{assu:FOBI}$ (v) $, the same holds for the sequence of the sample $k$-TJADE objective functions.

	We close the proof with the M-estimator convergence argument. In the following we use the notation, $ h(\textbf{W}, \textbf{Y}) := g(\textbf{H} \textbf{W}, \textbf{H} \textbf{Y})$ and $\hat{h}(\textbf{W}, \textbf{Y}_N) := \hat{g}(\hat{\textbf{H}} \textbf{W}, \hat{\textbf{H}} \textbf{Y}_N) $. As all our maximizers are unique only up to the signs and the order of their columns, to obtain a sequence such that  $\hat{\textbf{W}} \rightarrow_{\mathbb{P}}  \textbf{I}_p $, we restrict ourselves to a subset $\mathcal{U}_0$ of $ \mathcal{U} $ where the signs and order are fixed and $ \textbf{W} = \textbf{I}_p $ is the unique maximizer of $h(\textbf{W}, \textbf{Y})$. A corresponding set $\mathcal{U}_0$ can be constructed as follows,
	\begin{align*}
	\mathcal{U}_0 = \left\{ \textbf{W}\in \mathcal{R}^{p\times p} : \left(\textbf{W} \in \mathcal{U} \right) \land \left(\textbf{w}_i\tprime \textbf{1}_p \geq 0, \forall i\right) \land \left( \textbf{w}_1\tprime \textbf{D}^{11} \textbf{w}_1 \geq \ldots \geq \textbf{w}_p\tprime \textbf{D}^{pp} \textbf{w}_p \right) \right\},
	\end{align*}
	where $\textbf{D}^{ii} =(p-1) \textbf{E}^{ii}$. Moreover, as the conditions defining $ \mathcal{U}_0 $ are continuous in $ \textbf{W} $, there exists $ \varepsilon > 0 $ such that an $ \varepsilon $-neighborhood around $ \textbf{I}_p $ fits completely within $ \mathcal{U}_0 $. The following set defines the complement of such ball in $ \mathcal{U}_0 $,
	\begin{align*}
	\mathcal{U}_\varepsilon = \left\{ \textbf{W} \in \mathcal{U}_0 : \left\| \textbf{W} - \textbf{I}_p \right\|_\textnormal{F} \geq \varepsilon  >0   \right\}.
	\end{align*}
	Since there is a finite number of different combinations of the order and the signs of the columns of $\hat{\textbf{W}}$, we can, for every $ n $, consider the equivalent maximizer $\hat{\textbf{W}} \equiv \textbf{PJ}\hat{\textbf{W}}$ that belongs to the set $\mathcal{U}_0$, i.e. $\hat{\textbf{W}} = {\arg\!\max}_{\textbf{W} \in \mathcal{U}_0 } \{ \hat{h}(\textbf{W}, \textbf{Y}_N) \}$. Furthermore, by the definition of $ \mathcal{U}_0 $, the unique maximizer $ {\arg\!\max}_{\textbf{W} \in \mathcal{U}_0 } \{ h(\textbf{W}, \textbf{Y}) \} = \textbf{I}_p $.
	
	Now, 
	\begin{align*}
	&\hat{h}(\hat{\textbf{W}},\textbf{Y}_N) \geq \hat{h}({\textbf{I}_p},\textbf{Y}_N)  - h(\textbf{I}_p,\textbf{Y}) +  h(\textbf{I}_p,\textbf{Y}),
	\end{align*}
	and by subtracting ${h}(\hat{\textbf{W}},\textbf{Y}) $ from both sides and applying Eq. \eqref{eq:supconv}, we obtain,
	\begin{align}
	\label{eq:supconv2}
	\begin{split}
	0 \leq h(\textbf{I}_p,\textbf{Y}) - {h}(\hat{\textbf{W}},\textbf{Y}) & \leq \hat{h}(\hat{\textbf{W}},\textbf{Y}_N)  - {h}(\hat{\textbf{W}},\textbf{Y}) + h(\textbf{I}_p,\textbf{Y}) - \hat{h}({\textbf{I}_p},\textbf{Y}_N) \\
	&\leq 
	2 \underset{\textbf{W} \in \mathcal{U}_0} {\sup}\left\{\left| \hat{h}({\textbf{W}},\textbf{Y}_N)  - {h}({\textbf{W},\textbf{Y}}) \right| \right\} = o_p(1).
	\end{split}
	\end{align}
	Furthermore, it follows from the uniqueness of the maximizer $ \textbf{I}_p $ of $ h $ in $ \mathcal{U}_0 $, that, for every $\varepsilon >0$, there exists  $\delta >0$, such that
	\begin{align*}
	\underset{\textbf{W} \in \mathcal{U}_\varepsilon }{\sup}\left\{ h(\textbf{W},\textbf{Y}) \right\} < h(\textbf{I}_p,\textbf{Y}) - \delta.
	\end{align*}
	Fix now $ \varepsilon > 0 $ small enough along with the corresponding $ \delta >0$. Then,
	\[
	\{ \hat{\textbf{W}} \in \mathcal{U}_\varepsilon \} \subset \{ h(\textbf{I}_p,\textbf{Y}) - h(\hat{{\textbf{W}}},\textbf{Y})  > \delta  \},
	\]
	and
	\begin{align*}
	\mathbb{P}\left[ \| \hat{\textbf{W}} - \textbf{I}_p \|_\textnormal{F} \geq \varepsilon \right] = \mathbb{P}\left[  \hat{\textbf{W}} \in \mathcal{U}_\varepsilon \right] 
	\leq  \mathbb{P}\left[\left\{h(\textbf{I}_p,\textbf{Y}) - h(\hat{\textbf{W}},\textbf{Y})  > \delta \right\} \right] \underset{n \rightarrow \infty}{\longrightarrow} 0,
	\end{align*}
	where the convergence follows from Eq. \eqref{eq:supconv2}. Thus, $\hat{\textbf{W}} \rightarrow_{\mathbb{P}}  \textbf{I}_p$. This concludes the proof.

\end{proof}

\begin{proof}[Proof of Theorem \ref{theo:limiting}]

	As stated in the proof of Theorem \ref{theo:consistency}, the estimated left $k$-TJADE transformation is $\hat{\boldsymbol{\Gamma}}{}^k = \hat{\textbf{V}}{}\tprime \hat{\textbf{H}} \hat{\boldsymbol{\Sigma}}{}_1^{-1/2} = \hat{\textbf{W}}{}\tprime \hat{\boldsymbol{\Sigma}}{}_1^{-1/2}$, where $\hat{\textbf{W}} = (\hat{\textbf{w}}_1, \ldots, \hat{\textbf{w}}_p)$ is the sequence of the maximizers of $\hat{g}(\hat{\textbf{H}} \textbf{W}, \hat{\textbf{H}} \textbf{Y}_N)$. It follows from linearization and the consistency of $ \hat{{\textbf{W}}} $ that,
	\begin{align*}
	\sqrt{n} (\hat{\boldsymbol{\Gamma}}{}^k - \textbf{I}_p ) &= \hat{\textbf{W}}{}\tprime \sqrt{n} (\hat{\boldsymbol{\Sigma}}{}_1^{-1/2} - \textbf{I}_p ) +  \sqrt{n}(\hat{\textbf{W}}{}\tprime - \textbf{I}_p) \\
	&= \sqrt{n} (\hat{\boldsymbol{\Sigma}}{}_1^{-1/2} - \textbf{I}_p ) + \sqrt{n}(\hat{\textbf{W}}{}\tprime - \textbf{I}_p) + o_p(1).
	\end{align*}
	As the covariance standardization part is the same in both $k$-TJADE and TJADE, our goal in the following is to show that the expression $\sqrt{n}(\hat{\textbf{W}}{}\tprime - \textbf{I}_p)$ is asymptotically equivalent to $\sqrt{n}(\hat{\textbf{W}}{}_\textnormal{J}\tprime - \textbf{I}_p)$, where $\hat{\textbf{W}}_\textnormal{J}$ is the regular TJADE-rotation.
	
	The estimating equations for $\hat{\textbf{W}}$ can be found by applying the technique of Lagrangian multipliers, under the constraint $\textbf{W}\tprime \textbf{W} = \textbf{I}_p$, to the objective function $\hat{g}(\hat{\textbf{H}} \textbf{W}, \hat{\textbf{H}} \textbf{Y}_N)$. Similarly as in \cite{miettinen2016separation}, the estimating equations are
	\[
	\hat{\textbf{w}}_l\tprime \hat{\textbf{T}} (\hat{\textbf{w}}_m) = \hat{\textbf{w}}_m\tprime \hat{\textbf{T}} (\hat{\textbf{w}}_l),
	\]
	for any two distinct columns of $\hat{\textbf{W}}$, $m \neq l$, where
	\begin{align}\label{eq:estimating_eq_part_1}
	\begin{split}
	\hat{\textbf{T}} (\hat{\textbf{w}}_m) &= \sum_{|i - j| < k} \hat{\textbf{H}}\tprime \hat{\textbf{C}}{}^{ij}[\hat{\textbf{H}} \textbf{Y}_N] \hat{\textbf{H}} \hat{\textbf{w}}_m \hat{\textbf{w}}_m\tprime \hat{\textbf{H}}\tprime \hat{\textbf{C}}{}^{ij}[\hat{\textbf{H}} \textbf{Y}_N] \hat{\textbf{H}} \hat{\textbf{w}}_m \\
	&= \sum_{|i - j| < k} \sum_{a, b} \sum_{a', b'} \hat{h}_{ia} \hat{h}_{jb} \hat{h}_{ia'} \hat{h}_{jb'} \hat{\textbf{C}}{}^{ab}[ \textbf{Y}_N ] \hat{\textbf{w}}_m \hat{\textbf{w}}_m\tprime \hat{\textbf{C}}{}^{a'b'}[ \textbf{Y}_N ] \hat{\textbf{w}}_m \\
	&= \sum_{a, b} \hat{\textbf{C}}{}^{ab}[ \textbf{Y}_N] \hat{\textbf{w}}_m \hat{\textbf{w}}_m\tprime \hat{\textbf{C}}{}^{ab}[ \textbf{Y}_N ] \hat{\textbf{w}}_m - \hat{\textbf{m}}.
	\end{split}
	\end{align}
	In the above formula, the vector $ \hat{\textbf{m}} $ is defined as,
	\begin{align}\label{eq:estimating_eq_part_2}
	\hat{\textbf{m}} =  \sum_{|i - j| \geq k} \sum_{a, b} \sum_{a', b'} \hat{h}_{ia} \hat{h}_{jb} \hat{h}_{ia'} \hat{h}_{jb'} \hat{\textbf{C}}{}^{ab}[ \textbf{Y}_N ] \hat{\textbf{w}}_m \hat{\textbf{w}}_m\tprime \hat{\textbf{C}}{}^{a'b'}[ \textbf{Y}_N ] \hat{\textbf{w}}_m ,
	\end{align}
	the second equality is based on Lemma \ref{lem:H_simplification} and the third equality is based on the sum-completion technique that we applied also in the proof of Theorem \ref{theo:consistency}. Under Assumption \ref{assu:FOBI}$ (v) $, the vector $ \hat{\textbf{m}} $ satisfies $ \hat{\textbf{m}} = \mathcal{O}_p(1/n)$. We omit the proof, since it is almost identical to the proof of property $(1^\circ)$ in the proof of Theorem \ref{theo:consistency}.
	
	We denote $\textbf{C}^{ab} \coloneqq \textbf{C}^{ab}[ \textbf{Y} ]$ and $\hat{\textbf{C}}^{ab} \coloneqq \hat{\textbf{C}}^{ab}[ \textbf{Y}_N ]$. \red{Then, by Eq. \eqref{eq:estimating_eq_part_1} and Eq. \eqref{eq:estimating_eq_part_2}, any solution to the ($\sqrt{n}$-multiplied) $ k $-TJADE estimating equations,}
	\[
	\sqrt{n} \hat{\textbf{w}}_l\tprime \hat{\textbf{T}} (\hat{\textbf{w}}_m) = \sqrt{n} \hat{\textbf{w}}_m\tprime \hat{\textbf{T}} (\hat{\textbf{w}}_l), \quad m \neq l,
	\]
	\red{is also a solution to the estimating equations,}
	\begin{align}\label{eq:estimating_eq_jade}
	\sqrt{n} \sum_{a, b} \hat{\textbf{w}}_l\tprime \hat{\textbf{C}}{}^{ab}\hat{\textbf{w}}_m \hat{\textbf{w}}_m\tprime \hat{\textbf{C}}{}^{ab} \hat{\textbf{w}}_m = \sqrt{n} \sum_{a, b} \hat{\textbf{w}}_m\tprime \hat{\textbf{C}}{}^{ab} \hat{\textbf{w}}_l \hat{\textbf{w}}_l\tprime \hat{\textbf{C}}{}^{ab} \hat{\textbf{w}}_l + \mathcal{O}_p(1/\sqrt{n}),
	\end{align}
	\red{for $ m \neq l $, and vice versa.} (Note that we may also include the cases where $ m = l $ as those simply correspond to the trivial case $ 0 = o_p(1)$.) These are the estimating equations of the regular TJADE-rotation $\hat{\textbf{W}}_\textnormal{J}$, and yield a limiting distribution under Assumption \ref{assu:JADE}. It now follows that the limiting expression for $\hat{\textbf{W}}$ is asymptotically equivalent to the limiting expression of the TJADE-rotation $\hat{\textbf{W}}_\textnormal{J}$ and, consequently, under Assumption \ref{assu:JADE} and Assumption \ref{assu:FOBI}$ (v) $ we have that
	\[
	\sqrt{n} (\hat{\boldsymbol{\Gamma}}{}^{k} - \textbf{I}_{p} ) = \sqrt{n} (\hat{\boldsymbol{\Gamma}}{}^\textnormal{J} - \textbf{I}_{p} ) + o_p(1),
	\]
	technically concluding the proof. However, as the derivation of the convergence rate, root-$ n $, is left implicit in \cite{virta2017jade}, we present it here for completeness.
	
	Letting $ \hat{\textbf{D}}{}^{ab} = \hat{\textbf{W}}{}\tprime \hat{\textbf{C}}^{ab} \hat{\textbf{W}} \rightarrow_{\mathbb{P}} \delta_{ab} \kappa_a \textbf{E}^{aa} $, the estimating equations in Eq. \eqref{eq:estimating_eq_jade} along with the orthogonality constraint can be written in matrix form as,
	\begin{align*}
	\sqrt{n}\sum_{a,b} \hat{\textbf{W}}\tprime \hat{\textbf{C}}^{ab}  \hat{\textbf{W}} \mathrm{diag}(\hat{\textbf{D}}^{ab}) &= \sqrt{n}\sum_{a,b} \mathrm{diag}(\hat{\textbf{D}}^{ab}) \hat{\textbf{W}}\tprime \hat{\textbf{C}}^{ab}  \hat{\textbf{W}} + o_p(1), \label{eq:estimating_eq_1} \numberthis \\
	\sqrt{n} ( \hat{\textbf{W}}\tprime \hat{\textbf{W}} - \textbf{I}_p ) &= \textbf{0}. \label{eq:estimating_eq_2} \numberthis
	\end{align*}
	In order to prove that the convergence rate is root-$ n $, we bring the above estimating equations to the forms,
	\begin{align}\label{eq:two_a_equations}
	\begin{split}
	\hat{\textbf{A}}_1 \sqrt{n} \mathrm{vec}( \hat{\textbf{W}} - \textbf{I}_p ) &= \mathcal{O}_p(1),\\
	\hat{\textbf{A}}_2 \sqrt{n} \mathrm{vec}( \hat{\textbf{W}} - \textbf{I}_p ) &= \mathcal{O}_p(1),
	\end{split}
	\end{align}
	where $ \hat{\textbf{A}}_1, \hat{\textbf{A}}_2 $ converge in probability to some constant matrices $ \textbf{A}_1, \textbf{A}_2$ and the right-hand sides have limiting distributions depending only on the matrices $ \hat{\textbf{C}}{}^{ab} $, $ a,b \in \{1, \ldots , p\} $. By the central limit theorem, the limiting distributions of $ \sqrt{n} (\hat{\textbf{C}}{}^{ab} - \textbf{C}{}^{ab}) $ are multivariate normal. We also show that while neither $ \textbf{A}_1 $ nor $ \textbf{A}_2$ is invertible, their sum is, and we may take the sum of the two equations in Eq. \eqref{eq:two_a_equations} and multiply both sides of the resulting equation by $ ( \hat{\textbf{A}}_1 + \hat{\textbf{A}}_2 )^{-1}$ from the left. Applying the continuous mapping theorem and Slutsky's theorem then gives us the desired limiting result.
	
	We begin with the orthogonality constraint in Eq. \eqref{eq:estimating_eq_2} and expand as
	\[ 
	\textbf{0} = \sqrt{n} (\hat{\textbf{W}}\tprime \hat{\textbf{W}} - \textbf{I}_p) = \sqrt{n} (\hat{\textbf{W}}\tprime - \textbf{I}_p) \hat{\textbf{W}} + \sqrt{n} (\hat{\textbf{W}} - \textbf{I}_p). 
	\]
	Applying the vectorization operator along with the identity $ \mathrm{vec}(\textbf{A} \textbf{X} \textbf{B}\tprime) = (\textbf{B} \otimes \textbf{A}) \mathrm{vec}(\textbf{X})$ gives the first desired equation,
	\[ 
	\left[ (\hat{\textbf{W}}\tprime \otimes \textbf{I}_p ) \textbf{K} + \textbf{I}_{p^2} \right] \sqrt{n} \mathrm{vec}(  \hat{\textbf{W}} - \textbf{I}_p ) = \textbf{0},
	\]
	where $ \textbf{K} = \sum_{k = 1}^p \sum_{\ell = 1}^p \textbf{E}^{k\ell} \otimes \textbf{E}^{\ell k}$ is the commutation matrix. Note that $ \textbf{K} \mathrm{vec}(\textbf{A}\tprime) = \mathrm{vec}(\textbf{A}) $ for any matrix $ \textbf{A} \in \mathcal{R}^{p \times p} $. Thus $ \hat{\textbf{A}}_1 = (\hat{\textbf{W}}{}\tprime \otimes \textbf{I}_p ) \textbf{K} + \textbf{I}_{p^2} \rightarrow_{\mathbb{P}} \textbf{K} + \textbf{I}_{p^2} = \textbf{A}_1$.
	
	Let now $ \hat{\textbf{Y}} = \sum_{ab} \hat{\textbf{W}}{}\tprime \hat{\textbf{C}}{}^{ab}  \hat{\textbf{W}} \mathrm{diag}(\hat{\textbf{D}}{}^{ab}) $. Eq. \eqref{eq:estimating_eq_1} gives that $ \sqrt{n} \hat{\textbf{Y}} $ is asymptotically symmetric. That is, $ \sqrt{n} \hat{\textbf{Y}} = \sqrt{n} \hat{\textbf{Y}}{}\tprime + o_p(1) $. By applying $\hat{\textbf{C}}{}^{ab} = \hat{\textbf{C}}{}^{ab} - \delta_{ab} \kappa_a \textbf{E}^{aa} + \delta_{ab} \kappa_a \textbf{E}^{aa} $ and Slutsky's theorem, we obtain
	\begin{align*}
	\sqrt{n}\hat{\textbf{Y}} &=  \sum_{ab} \hat{\textbf{W}}\tprime \sqrt{n} ( \hat{\textbf{C}}^{ab} - \delta_{ab} \kappa_a \textbf{E}^{aa} )  \hat{\textbf{W}} \mathrm{diag}(\hat{\textbf{D}}^{ab}) + \sum_{ab} \hat{\textbf{W}}\tprime \sqrt{n} \delta_{ab} \kappa_a \textbf{E}^{aa} \hat{\textbf{W}} \mathrm{diag}(\hat{\textbf{D}}^{ab})\\
	&= \sum_a \kappa_a \sqrt{n} ( \hat{\textbf{C}}^{aa} - \kappa_a \textbf{E}^{aa} ) \textbf{E}^{aa} + \sum_{a} \sqrt{n} \kappa_a  \hat{\textbf{W}}\tprime \textbf{E}^{aa} \hat{\textbf{W}} \mathrm{diag}(\hat{\textbf{D}}^{aa}) + o_p(1).
	\end{align*}
	Let $ \sqrt{n} \hat{\textbf{F}}_1 = \sum_a \kappa_a \sqrt{n} ( \hat{\textbf{C}}{}^{aa} - \kappa_a \textbf{E}^{aa} ) \textbf{E}^{aa} $. Applying $ \hat{\textbf{D}}^{aa} = \hat{\textbf{W}}{}\tprime \hat{\textbf{C}}{}^{aa} \hat{\textbf{W}} = \hat{\textbf{W}}{}\tprime ( \hat{\textbf{C}}{}^{aa} - \kappa_a \textbf{E}^{aa}) \hat{\textbf{W}} + \kappa_a \hat{\textbf{W}}{}\tprime \textbf{E}^{aa} \hat{\textbf{W}} $, we obtain
	\begin{align*}
	\sqrt{n}\hat{\textbf{Y}} &= \sum_{a} \kappa_a \hat{\textbf{W}}\tprime \textbf{E}^{aa} \hat{\textbf{W}} \mathrm{diag}\left( \hat{\textbf{W}}\tprime \sqrt{n} ( \hat{\textbf{C}}^{aa} - \kappa_a \textbf{E}^{aa}) \hat{\textbf{W}} \right) \\
	&+ \sqrt{n} \sum_{a} \kappa_a \hat{\textbf{W}}\tprime \textbf{E}^{aa} \hat{\textbf{W}} \mathrm{diag}\left( \kappa_a \hat{\textbf{W}}\tprime  \textbf{E}^{aa} \hat{\textbf{W}} \right) + \sqrt{n} \hat{\textbf{F}}_1 + o_p(1),\\
	&= \sum_{a}\kappa_a \textbf{E}^{aa} \mathrm{diag}\left( \sqrt{n} ( \hat{\textbf{C}}^{aa} - \kappa_a \textbf{E}^{aa}) \right) \\
	&+ \sqrt{n} \sum_{a} \kappa_a^2 \hat{\textbf{W}}\tprime   \textbf{E}^{aa} \hat{\textbf{W}} \mathrm{diag}\left( \hat{\textbf{W}}\tprime \textbf{E}^{aa} \hat{\textbf{W}} \right) + \sqrt{n} \hat{\textbf{F}}_1 + o_p(1)
	\end{align*}
	Let $ \sqrt{n} \hat{\textbf{F}}_2 = \sum_{a}\kappa_a \textbf{E}^{aa} \mathrm{diag}( \sqrt{n} ( \hat{\textbf{C}}{}^{aa} - \kappa_a \textbf{E}^{aa}) ) $. By plugging in $ \hat{{\textbf{W}}} = \hat{{\textbf{W}}} - \textbf{I}_p + \textbf{I}_p $, we obtain
	\begin{align*}
	\sqrt{n}\hat{\textbf{Y}} &= \sqrt{n} \sum_{a} \kappa_a^2  \hat{\textbf{W}}\tprime  \textbf{E}^{aa} \hat{\textbf{W}} \mathrm{diag}\left( \hat{\textbf{W}}\tprime \textbf{E}^{aa} \hat{\textbf{W}} \right) + \sqrt{n} \hat{\textbf{F}}_1 + \sqrt{n} \hat{\textbf{F}}_2 +  o_p(1)\\
	&= \sum_{a} \kappa_a^2 \sqrt{n} ( \hat{\textbf{W}}\tprime - \textbf{I}_p)  \textbf{E}^{aa} \hat{\textbf{W}} \mathrm{diag}\left( \hat{\textbf{W}}\tprime \textbf{E}^{aa} \hat{\textbf{W}} \right)\\
	&+ \sum_{a} \kappa_a^2  \textbf{E}^{aa}  \sqrt{n} ( \hat{\textbf{W}} - \textbf{I}_p) \mathrm{diag}\left( \hat{\textbf{W}}\tprime \textbf{E}^{aa} \hat{\textbf{W}} \right) \\
	&+ \sqrt{n} \sum_{a} \kappa_a^2  \textbf{E}^{aa}  \mathrm{diag}\left( \hat{\textbf{W}}\tprime \textbf{E}^{aa} \hat{\textbf{W}} \right)  + \sqrt{n} \hat{\textbf{F}}_1 + \sqrt{n} \hat{\textbf{F}}_2 +  o_p(1).
	\end{align*}
	The terms $ \sqrt{n} \sum_{a} \kappa_a^2  \textbf{E}^{aa}  \mathrm{diag}( \hat{\textbf{W}}{}\tprime \textbf{E}^{aa} \hat{\textbf{W}} ) $ and $ \sqrt{n} \hat{\textbf{F}}_2 $ are symmetric and cancel themselves out in the symmetry identity $\sqrt{n} \hat{\textbf{Y}} = \sqrt{n} \hat{\textbf{Y}}{}\tprime + o_p(1) $, which can also be written as $ (\textbf{I}_{p^2} - \textbf{K}) \sqrt{n} \mathrm{vec}( \hat{\textbf{Y}}) = o_p(1)$. Let $ \hat{\textbf{G}}{}^{aa} = \mathrm{diag}( \hat{\textbf{W}}{}\tprime \textbf{E}^{aa} \hat{\textbf{W}} ) \rightarrow_{\mathbb{P}} \textbf{E}^{aa}$ and vectorize to obtain
	\begin{align*}
	& \left(\textbf{I}_{p^2} - \textbf{K}\right) \left( \sum_a \kappa_a^2 [\hat{\textbf{G}}^{aa} \hat{\textbf{W}}\tprime \textbf{E}^{aa} \otimes \textbf{I}_p ] \textbf{K} +  \sum_a \kappa_a^2 [\hat{\textbf{G}}^{aa} \otimes \textbf{E}^{aa} ]  \right) \sqrt{n} \mathrm{vec}(  \hat{\textbf{W}} - \textbf{I}_p )\\
	=& - \left(\textbf{I}_{p^2} - \textbf{K}\right) \sqrt{n} \mathrm{vec}( \hat{\textbf{F}}_1 ) + o_p(1).
	\end{align*}
	Thus the matrix $ \hat{\textbf{A}}_2 $ in Eq. \eqref{eq:two_a_equations} can be given as 
	\begin{align*}
	\hat{\textbf{A}}_2 &= \left(\textbf{I}_{p^2} - \textbf{K}\right)\left( \sum_a \kappa_a^2 [\hat{\textbf{G}}^{aa} \hat{\textbf{W}}\tprime \textbf{E}^{aa} \otimes \textbf{I}_p ] \textbf{K} +  \sum_a \kappa_a^2 [\hat{\textbf{G}}^{aa} \otimes \textbf{E}^{aa} ]  \right) \\
	& \rightarrow_{\mathbb{P}}  \left(\textbf{I}_{p^2} - \textbf{K}\right) \left( \sum_a \kappa_a^2 [ \textbf{E}^{aa} \otimes \textbf{I}_p ] \textbf{K} +  \sum_a \kappa_a^2 [\textbf{E}^{aa} \otimes \textbf{E}^{aa} ]  \right) \\
	&=  \left(\textbf{I}_{p^2} - \textbf{K}\right) \sum_a \kappa_a^2 [ \textbf{E}^{aa} \otimes \textbf{I}_p ] \textbf{K}\\
	&= \sum_a \kappa_a^2 [ \textbf{E}^{aa} \otimes \textbf{I}_p ] \textbf{K} - \sum_a \kappa_a^2 [ \textbf{I}_p \otimes \textbf{E}^{aa} ]\\
	&= [\textbf{D} \otimes \textbf{I}_p] \textbf{K} - [\textbf{I}_p \otimes \textbf{D} ] = \textbf{A}_2,
	\end{align*}
	where we have used the identities $(\textbf{I}_{p^2} - \textbf{K}) \sum_a \kappa_a^2 [\textbf{E}^{aa} \otimes \textbf{E}^{aa} ] = \textbf{0}$ and $ \textbf{K} [ \textbf{E}^{aa} \otimes \textbf{I}_p ] \textbf{K} =  [ \textbf{I}_p \otimes \textbf{E}^{aa} ]$ and the notation $ \textbf{D} = \sum_a \kappa_a^2  \textbf{E}^{aa} $.
	
	Taking the sum over the two expanded estimating equations in Eq. \eqref{eq:two_a_equations}, we now obtain,
	\begin{align}\label{eq:expansion}
	(\hat{\textbf{A}}_1 + \hat{\textbf{A}}_2) \sqrt{n} \mathrm{vec}(  \hat{\textbf{W}} - \textbf{I}_p ) =  - \left(\textbf{I}_{p^2} - \textbf{K}\right)  \sqrt{n} \mathrm{vec}( \hat{\textbf{F}}_1 ) + o_p(1),
	\end{align}
	where $ \hat{\textbf{A}}_1 + \hat{\textbf{A}}_2 \rightarrow_{\mathbb{P}} \textbf{A} = {\textbf{A}}_1 + {\textbf{A}}_2 = [\textbf{D} \otimes \textbf{I}_p] \textbf{K} - [\textbf{I}_p \otimes \textbf{D} ] + \textbf{K} + \textbf{I}_{p^2} $. The right-hand side of Eq. \eqref{eq:expansion} equals
	\begin{align*} 
	\left(\textbf{K} - \textbf{I}_{p^2} \right)  \sqrt{n} \mathrm{vec}( \hat{\textbf{F}}_1 ) + o_p(1) = \sqrt{n} \mathrm{vec}( \hat{\textbf{F}}_1{}\tprime  -   \hat{\textbf{F}}_1 ) + o_p(1) \label{eq:right_hand_side} \numberthis,
	\end{align*}
	which is $ \mathcal{O}_p(1) $ as $ \sqrt{n} \hat{\textbf{F}}_1 = \mathcal{O}_p(1) $. To see that $ \textbf{A} $ is invertible, we compute its determinant and verify that it is non-zero under Assumption \ref{assu:JADE}. Recall that Assumption \ref{assu:JADE} says that there is maximally one zero diagonal element in $ \textbf{D} $.
	
	Observe first that in the case $ p = 2 $, the matrix $ \textbf{A} $ takes the form,
	\[ 
	\begin{pmatrix}
	2 & 0 & 0 & 0\\
	0 & 1 - \kappa_2^2 & 1 + \kappa_1^2 & 0\\
	0 & 1 + \kappa_2^2 & 1 - \kappa_1^2 & 0\\
	0 & 0 & 0 & 2
	\end{pmatrix},
	\]
	where the $ 2 \times 2 $ diagonal block has determinant equal to $ -2(\kappa_1^2 + \kappa_2^2) $. As the determinant of a block diagonal matrix is the product of the determinants of the blocks, we have that $ \mathrm{det}(\textbf{A}) = -8(\kappa_1^2 + \kappa_2^2)$, verifying our claim in the case $ p = 2 $.
	
	For general $ p $, the matrix $ \textbf{A} $ can, after a suitable permutation of its rows and columns, not affecting the value of the determinant, be seen to consist of similar blocks as in the case $ p = 2 $. For each diagonal element of $ \sqrt{n} ( \hat{\textbf{W}} - \textbf{I}_p) $ (the elements $ a + (a - 1) p $, $ a \in \{ 1, \ldots , p\} $, of $\sqrt{n} \mathrm{vec}(  \hat{\textbf{W}} - \textbf{I}_p )$), we get a $ 1 \times 1 $ diagonal block equal to $ 2 $. For each $ (a, b) $th off-diagonal element in the lower triangle of $ \textbf{A} $, we get a $ 2 \times 2 $ diagonal block,
	\begin{align}\label{eq:2_by_2_block}
	\begin{pmatrix}
	1 - \kappa_b^2 & 1 + \kappa_a^2 \\
	1 + \kappa_b^2 & 1 - \kappa_a^2 \\
	\end{pmatrix},
	\end{align}
	with the determinant equal to $ -2(\kappa_a^2 + \kappa_b^2) $. As the number of diagonal elements is $ p $ and the number of off-diagonal elements in the lower triangle is $ p(p - 1)/2 $, the total determinant is,
	\[ 
	\mathrm{det}(\textbf{A}) = 2^p  (-2)^{p(p - 1)/2} \prod_{a > b} (\kappa_a^2 + \kappa_b^2),
	\]
	which is non-zero if and only if Assumption \ref{assu:JADE} holds.
	
	Multiply now both sides of Eq. \eqref{eq:expansion} by $ (\hat{\textbf{A}}_1 + \hat{\textbf{A}}_2)^{-1} $ (which is asymptotically well-defined as $\mathrm{det}(\textbf{A})  \neq 0$) and invoke Slutsky's theorem to obtain,
	\begin{align*}
	\sqrt{n} \mathrm{vec}(  \hat{\textbf{W}} - \textbf{I}_p ) = (\hat{\textbf{A}}_1 + \hat{\textbf{A}}_2)^{-1}  \sqrt{n} \mathrm{vec}( \hat{\textbf{F}}_1{}\tprime  -   \hat{\textbf{F}}_1 )   + o_p(1),
	\end{align*}
	where the right-hand side has the same limiting distribution as $ \textbf{A}^{-1} \sqrt{n} \mathrm{vec}( \hat{\textbf{F}}_1{}\tprime  -   \hat{\textbf{F}}_1 ) $, i.e., a multivariate normal distribution. This reveals that $ \sqrt{n}(  \hat{\textbf{W}} - \textbf{I}_p ) $ is indeed $ \mathcal{O}_p(1) $. To obtain asymptotic expressions for the elements of $ \sqrt{n}(  \hat{\textbf{W}} - \textbf{I}_p ) $, we inspect its diagonal and off-diagonal elements separately.
	
	Each element of $\sqrt{n}\mathrm{vec}(  \hat{\textbf{W}} - \textbf{I}_p )$ corresponding to a diagonal element of $ \sqrt{n}(  \hat{\textbf{W}} - \textbf{I}_p )$ is associated with a $ 1 \times 1 $ diagonal block in $ \textbf{A}^{-1} $ equal to $ 1/2 $. This block picks the corresponding diagonal element from the right-hand side of Eq. \eqref{eq:right_hand_side} into the expression of $ \sqrt{n}(  \hat{w}_{aa} - 1 ) $. However, the diagonal elements of the expression inside the vectorization operator in Eq. \eqref{eq:right_hand_side} are clearly zero and we have,
	\begin{align}\label{eq:result_diag} 
	\sqrt{n}(  \hat{w}_{aa} - 1 ) = o_p(1), \quad \forall a \in \{1, \ldots , p\}.
	\end{align}
	
	Take now a pair of off-diagonal elements $ (a, b) $, $ (b, a) $, with $ a > b $. Each such pair has a $ 2 \times 2 $ matrix of the form of Eq. \eqref{eq:2_by_2_block} associated with it, with the upper row corresponding to the element $ (a, b) $. By Cramer's rule, the inverse of the matrix is
	\[ 
	\frac{1}{2(\kappa_a^2 + \kappa_b^2)}
	\begin{pmatrix}
	\kappa_a^2 - 1 & \kappa_a^2 + 1 \\
	\kappa_b^2 + 1 & \kappa_b^2 - 1 \\
	\end{pmatrix}.
	\]
	The corresponding rows in $ \textbf{A}^{-1} \sqrt{n} \mathrm{vec}( \hat{\textbf{F}}_1{}\tprime  -   \hat{\textbf{F}}_1 ) $ can be given as
	\begin{align*}
	\frac{1}{2(\kappa_a^2 + \kappa_b^2)}
	\begin{pmatrix}
	\kappa_a^2 - 1 & \kappa_a^2 + 1 \\
	\kappa_b^2 + 1 & \kappa_b^2 - 1 \\
	\end{pmatrix}
	\begin{pmatrix}
	[\hat{\textbf{F}}_1{}\tprime  -   \hat{\textbf{F}}_1]_{ab} \\
	[\hat{\textbf{F}}_1{}\tprime  -   \hat{\textbf{F}}_1]_{ba} \\
	\end{pmatrix} =
	\frac{1}{(\kappa_a^2 + \kappa_b^2)} \begin{pmatrix}
	- [\hat{\textbf{F}}_1{}\tprime  -   \hat{\textbf{F}}_1]_{ab} \\
	[\hat{\textbf{F}}_1{}\tprime  -   \hat{\textbf{F}}_1]_{ab}
	\end{pmatrix},
	\end{align*}
	where $ [\textbf{M}]_{ab} $ denotes the $ (a, b) $ element of a matrix $ \textbf{M} $ and we have used  $ [\hat{\textbf{F}}_1{}\tprime  -   \hat{\textbf{F}}_1]_{ba} = - [\hat{\textbf{F}}_1{}\tprime  -   \hat{\textbf{F}}_1]_{ab}$ to simplify the expression. As $ [\hat{\textbf{F}}_1{}\tprime  -   \hat{\textbf{F}}_1]_{ab} = \kappa_a \sqrt{n}\hat{\textbf{C}}{}^{aa}_{ba} - \kappa_b \sqrt{n}\hat{\textbf{C}}{}^{bb}_{ab} = \kappa_a \sqrt{n}\hat{\textbf{C}}{}^{aa}_{ab} - \kappa_b \sqrt{n}\hat{\textbf{C}}{}^{bb}_{ba}$, we obtain
	\begin{align}\label{eq:result_off_diag} 
	\sqrt{n} \hat{w}_{ab} &= \frac{\kappa_a \sqrt{n}\hat{\textbf{C}}^{aa}_{ab} - \kappa_b \sqrt{n}\hat{\textbf{C}}^{bb}_{ba}}{\kappa_a^2 + \kappa_b^2} + o_p(1) 
	\end{align}
	for all $ a \neq b $, $ a, b \in \{1, \ldots , p\} $. The expressions in Eqs. \eqref{eq:result_diag} and \eqref{eq:result_off_diag} now match exactly to those obtained in the proof of \cite[Theorem 2]{virta2017jade} (with a different $ o_p(1) $-sequence, though) and the limiting variances can be derived as in \cite[Corollary 1]{virta2017jade}.
	
\end{proof}

%
%
%
%
%
%


\bibliographystyle{chicago}
\bibliography{references}

\end{document}